\documentclass[12pt]{amsart}
\usepackage{amssymb,latexsym,amsmath,amsthm,amscd}
\usepackage{setspace}
\usepackage[all]{xy} \xyoption{arc}
\usepackage[left=3cm,top=2cm,right=3cm,bottom = 2cm]{geometry}
\usepackage{graphicx}
\usepackage[usenames,dvipsnames]{color}
\usepackage{charter}

\theoremstyle{plain}
\newtheorem{thm}{Theorem}
\newtheorem{lem}[thm]{Lemma}
\newtheorem{prop}[thm]{Proposition}

\theoremstyle{definition}
\newtheorem{dfn}[thm]{Definition}
\newtheorem{ex}[thm]{Example}

\theoremstyle{remark}
\newtheorem{rmk}[thm]{Remark}

\newcommand{\cE}{\mathcal{E}}

\newcommand{\cH}{\mathcal{H}}

\newcommand{\cS}{\mathcal{S}}

\newcommand{\veps}{\varepsilon}

\DeclareMathOperator{\uhp}{\mathfrak{H}}

\DeclareMathOperator{\Tr}{Tr}

\DeclareMathOperator{\GL}{GL}

\DeclareMathOperator{\SL}{SL}

\newcommand*{\df}{\mathrel{\vcenter{\baselineskip0.5ex \lineskiplimit0pt
                     \hbox{\scriptsize.}\hbox{\scriptsize.}}} =}

\providecommand{\norm}[1]{\lVert#1\rVert}
\providecommand{\abs}[1]{\left\lvert#1\right\rvert}

\providecommand{\twomat}[4]{\left(\begin{array}{cc}#1&#2\\#3&#4\end{array}\right)}
\providecommand{\stwomat}[4]{\left(\begin{smallmatrix}#1&#2\\#3&#4\end{smallmatrix}\right)}

\providecommand{\pseries}[2]{#1[\![ #2 ]\!]}

\newcommand{\QQ}{\mathbf{Q}}
\newcommand{\CC}{\mathbf{C}}

\newcommand{\ZZ}{\mathbf{Z}}
\newcommand{\PP}{\mathbf{P}}
\newcommand{\RR}{\mathbf{R}}

\newcommand{\FF}{\mathbf{F}}

\providecommand{\floor}[1]{\lfloor #1 \rfloor}
\DeclareMathOperator{\Eis}{Eis}
\DeclareMathOperator{\ddiag}{diag}
\DeclareMathOperator{\PSL}{PSL}

\begin{document}
\title[Hypergeometric series, differential equations and modular forms]{Hypergeometric series, modular linear differential equations, and vector-valued modular forms}
\author[Franc]{Cameron Franc}
\email{cfranc@ucsc.edu}
\address{Department of Mathematics\\
University of California, Santa Cruz}
\author[Mason]{Geoffrey Mason}
\email{gem@ucsc.edu}
\address{Department of Mathematics\\
University of California, Santa Cruz}
\thanks{The second author is supported by the NSF}
\dedicatory{Dedicated to the memory of Marvin Isadore Knopp}
\date{}

\begin{abstract}
We survey the theory of vector-valued modular forms and their
connections with modular differential equations and Fuchsian equations over 
the three-punctured sphere.\ We present a number of numerical examples showing how
the theory in dimensions 2 and 3 leads naturally to close connections between
modular forms and hypergeometric series.


\medskip \noindent
MSC(2010)
\end{abstract}
\maketitle
\tableofcontents

\setlength{\parskip}{10pt plus 3pt minus 3pt}

\section{Introduction} \label{s:intro}

Renewed interest in a general study of vector-valued modular forms over the last $10$ years was in large part  due to the efforts of Marvin Knopp.\ Beginning with generalized modular forms \cite{KnoMas1}  and soon followed by  general vector-valued modular forms (vvmf) \cite{KnoMas2}, \cite{KnoMas3}, the authors leaned heavily on Knopp's deep and extensive knowledge of the classical German school of modular forms (Hecke, Petersson, Eichler and others) to show that the standard development based on Poincar\'{e} series could be carried through even in the vector-valued setting.\ Knopp also realized that an extension of Eichler cohomology to the setting of generalized modular forms was feasible, and he published several papers on the subject with Wissam Raji \cite{KnoRaj1}, \cite{KnoRaj2}.\footnote{It is interesting that Joseph Lehner is also a coauthor of the first of these papers.\ As Knopp once recounted to the second author, Lehner had long ago found a new approach to the classical Eichler cohomology theory based on an application of Stokes' theorem.\ Much later Knopp found that this method, which had never been published, fit well with the requirements for an Eichler cohomology of gmfs.\ Though Lehner had long since retired, Knopp was in no doubt that he should be a coauthor.} These trends led naturally to \emph{logarithmic} vvmfs and
connections with \emph{modular linear differential equations} (mldes).

This paper explains the relationship between vvmfs and mldes.\ The literature contains many papers dedicated to the study of particular differential equations with modular coefficients, but there seems to be no systematic development.\ Viewing such differential equations through the optic of vvmfs (and vice versa) turns out to be fruitful.\ In this approach, it is the \emph{monodromy} of the mlde that is fundamental.

We now discuss the contents of the present paper.\ It may have been Selberg \cite{Sel} who made the first nontrivial applications of vvmfs, though Gunning's work on Eichler cohomology \cite{Gun} appeared several years before Selberg's paper.\ In Section 2 we briefly discuss both of these developments, and in addition we describe more recent growth estimates for vvmfs whose monodromy representation does not necessarily factor through a finite quotient of the modular group.\ We go on to
explain the connections between vvmfs and differential equations, an idea that likely originated with Poincar\'{e} \cite{Poi}.\ In Section 3 we discuss the so-called free module theorem for vvmfs, explain the connection between mldes and Fuchsian differential equations on the three-punctured sphere, and show how these ideas may be combined to solve the Riemann-Hilbert problem by producing a vvmf with a given monodromy representation of $\SL_2(\ZZ)$.\ It turns out that in dimensions $2$ and $3$ the Fuchsian equations that arise are solved by hypergeometric series.\ This is the general theme of Section 4, and we give a number of detailed explicit numerical examples that illustrate the theory.

Topics that we do \emph{not} discuss in detail include arithmetic aspects of the theory of vector-valued modular forms (cf.\ \cite{GKZ}, \cite{Bor2}, \cite{Bor1}), generalized modular forms, and applications of vvmfs and mldes to conformal field theory (see e.g.\ \cite{KohMas}, \cite{MaMuS}, \cite{Tu}).\ Actually, it was the possibility of applications to conformal field theory that  largely  motivated the recent development of vvmf theory.\ An alternate development that is similarly motivated can be found in the papers of Bantay and Gannon (\cite{BanGan}, \cite{Gan}).

\section{Vector-valued modular forms}
\subsection{Definition and first properties}\label{SSDandP}
Let $\uhp$ denote the complex upper half plane, and let $\Gamma = \SL_2(\ZZ)$ with standard generators
\begin{eqnarray*}
S=\left(\begin{array}{cc}0 & -1 \\1 & 0\end{array}\right),\quad\quad T=\left(\begin{array}{cc}1 & 1 \\0 & 1\end{array}\right).
\end{eqnarray*}
Write $\bar \Gamma = \Gamma/\{\pm 1\}$.\ Let $\rho \colon \Gamma \to \GL_n(\CC)$ denote a finite-dimensional complex representation of $\Gamma$.\ If $F \colon \uhp \to \CC^n$ is a function, then let $F|_k\gamma$ denote the usual stroke operator of weight $k \in \ZZ$, for $\gamma \in \Gamma$.

\begin{dfn}\label{def1}
 A \emph{vector-valued modular form} of weight $k \in \ZZ$ relative to $\rho$ is a function $F \colon \uhp \to \CC^n$ satisfying the following three conditions:
\begin{eqnarray}
&& (a)\ F\ \mbox{is holomorphic}; \notag\\
 && (b)\ F|_k\gamma(\tau) = \rho(\gamma)F(\tau) \ \mbox{for all $\gamma\in \Gamma$}; \label{vvmfdef}\\
&& (c)\ \mbox{$F$ has a $q$-expansion}. \notag
\end{eqnarray}
\end{dfn}
We will frequently use the abbreviation  vvmf to mean \emph{vector-valued modular form}, and use the notation $(F, \rho)$ to denote a vvmf $F$ relative to $\rho$ when we wish to specify the representation.\ A more invariant version of (\ref{vvmfdef}) runs as follows.\ Let $V$ be a left $\CC\Gamma$-module furnishing the representation $\rho$;\ $V$ has a canonical complex structure, and a vvmf $(F, V)$ is a holomorphic map $F: \uhp\rightarrow V$ with $F|_k\gamma(\tau) = \gamma.F(\tau)$.\ We recover (\ref{vvmfdef}) upon choosing a basis for $V$.

The $q$-expansion condition (c) is more intricate than the classical case and requires some explanation.\ A pair $(F, \rho)$ is said to be a \emph{weak} vvmf if only conditions (a) and (b)  hold.\ Write $F$ in coordinates as $F=\ ^{t}(f_1, ..., f_n)$, so that (b) gives $(f_i|_k)\gamma(\tau)=\sum_j \rho(\gamma)_{ij}f_j$ for each $i = 1,\ldots, n$.\ We call the $f_i$ the \emph{components} of $F$.\ The components span a finite-dimensional space $V'$ of holomorphic functions in $\uhp$, and $V'$ is a right $\Gamma$-module with respect to the stroke operator.\ Conversely, given such a $V'$ and any spanning set $f_1,\dots, f_n$ of $V'$, there is a representation
$\rho$ of $\Gamma$ such that $^{t}(f_1, ..., f_n)$ satisfies (a) and (b).\ Two weak vvmfs $(F_1, \rho_1), (F_2, \rho_2)$ are
\emph{equivalent} if there is an invertible $n\times n$ matrix $A$ such that $(AF_1, A\rho_1 A^{-1}) = (F_2, \rho_2)$.\ In particular,
$\rho_1$ and $\rho_2$ are equivalent in the usual sense of representations.\ Given a weak vvmf $(F, \rho)$, by changing to an equivalent vvmf if necessary we may, and shall, assume  that $\rho(T)$ is in \emph{modified} Jordan canonical form, i.e., $\rho(T)$ is a block diagonal matrix with blocks of the form
\begin{eqnarray*}
\left(\begin{array}{cccc} \lambda &  &  &\\ \lambda & \ddots & & \\ & \ddots & \ddots &\\  &  & \lambda&\lambda\end{array}\right).
\end{eqnarray*}
By an argument (see Theorem 2.2 of \cite{KnoMas5}) generalizing the fact that periodic meromorphic functions in $\uhp$ have a $q$-expansion, the components of $F$ corresponding to this block have the form
\begin{eqnarray}\label{logvvmf}
\left(\begin{array}{c}h_1 \\ h_2+\tau h_1 \\ h_3+\tau h_2+ {\tau \choose 2}h_1\\ \vdots \\  h_m+\tau h_{m-1}+\cdots+{\tau \choose m-1}h_1\end{array}\right),
\end{eqnarray}
where $m$ is the block size and each $h_i$ has a (convergent) \emph{ordinary} $q$-expansion 
\begin{eqnarray*}\label{ordqexp}
h_i(\tau)=\sum_{n\in \mathbb{Z}} a_n(i)q^{n+\mu},\ \ \lambda = e^{2\pi i \mu}.
\end{eqnarray*}
If $\mu$ is \emph{real}, then  $F$ is called \emph{meromorphic, holomorphic or cuspidal} at $\infty$ if 
the $q$-expansion for each $h_i(\tau)$ has respectively only finitely many negative powers of $q$, no negative powers of $q$, or only positive powers of $q$.\ For general complex $\mu$ we apply this definition using the real parts $n+Re(\mu)$ of the exponents for each block.\  These definitions are independent of the choice of  $\mu$, which is only defined mod $\ZZ$.\ For example, the meaning of (3) in the definition of a holomorphic vvmf is just that $F$ is holomorphic at $\infty$ in this sense for each block.

The occurrence of $q$-expansions involving powers of $\log q$ is a significant complicating factor in the theory.\ Such $q$-expansions intervene whenever $\rho(T)$ is \emph{not} semisimple.\ In this paper we will be mainly (though not exclusively) concerned with the case when $\rho(T)$ \emph{is} semisimple, for example if $\rho(T)$ has \emph{finite order}, or if $\rho$ is equivalent to a unitary representation.\ In this case, a vvmf is equivalent to one for which the component functions are ordinary $q$-expansions. 

It is not easy to pin down the first appearance of vector-valued modular forms in the history of mathematics.\ The earliest reference to vvmfs that the authors are aware of can be found in Gunning's work \cite{Gun}.\ However, the results of Selberg \cite{Sel} were obtained several years prior to publication and they may predate \cite{Gun}.\ In any case, \cite{Sel} is of particular interest, as it contains a nontrivial application of vvmfs to the study of classical scalar modular forms.\ In this work Selberg gives the uniform estimate  $O(n^{\frac k2-\frac 15})$ for the Fourier coefficients of weight $k$ cusp-forms on any finite index subgroup 
$\Gamma'$ of $\Gamma$, regardless of whether $\Gamma'$ is a congruence subgroup or not.\ The method takes such a scalar modular form on $\Gamma'$ and applies the stroke operator with respect to coset representatives for $\Gamma'$ in $\Gamma$ to get a finite collection of translated functions.\ These translates form the components of a vvmf for the full modular group $\Gamma$, and Selberg succeeded in applying the Rankin-Selberg method to bound the Fourier coefficients of such a vvmf.\ Selberg's result improved the classical work of Hecke establishing estimates $O(n^{k/2})$ and $O(n^{k-1})$ for scalar cusp-forms and holomorphic forms respectively ($k\geq 3$).

Around the same time, Eichler, Shimura and others began a study of a related class of automorphic vector-valued functions (cf.\ \cite{Eic1}, \cite{Gun}, \cite{KugShi}, \cite{MatMur}, \cite{MatShi}, \cite{Shi1},  etc).\ One of their goals was to develop tools for computing dimension and trace formulas for spaces of scalar-valued modular forms.\ To briefly explain the occurence of vvmfs in these works, let $G \subseteq \SL_2(\RR)$ denote a discrete subgroup of finite covolume, and let $M_n \colon G \to \GL_{n+1}(\CC)$ denote the $n$th symmetric power of the standard representation of $G$.\ Write
\[
  L_n(z) = M_n\twomat 1z01 = \left(\begin{matrix}
1 & nz & \frac{n(n-1)}{2}z^2 & \cdots & z^n\\
0 & 1 & (n-1)z & \cdots & z^{n-1}\\
  & \vdots &&&\vdots\\
0 & 0 &0&\cdots &1.  
\end{matrix}
  \right)
\]
In \cite{KugShi}, an \emph{$M_n$-vector with respect to} $G$ is defined to be an $(n+1)$-tuple $F = (F_i)$ of functions on $\uhp$ satisfying three properties V1, V2, V3. The first two are easy to state:
\begin{enumerate}
\item[(V1)] each $F_i$ is meromorphic;
\item[(V2)] for every $g \in G$, one has $F\circ g = M_n(g)g$.
\end{enumerate}
The third condition is a little more technical.\ To state it, let $s \in \PP^1(\RR)$ be a cusp of $G$, set $g = \stwomat{-s}{1}{-1}{0}$ if $s \neq \infty$, and otherwise set $g = \stwomat 1001$.\ If $F$ satisfies V1 and V2 then there exist $n+1$ functions $f_0(q),\ldots, f_n(q)$ meromorphic in $0 < \abs q < 1$ such that
\[
  L_n(z)^{-1}M_n(g)^{-1}(F\circ g) = \left(\begin{matrix}
  f_0(q)\\ \vdots\\ f_n(q)
  \end{matrix}
  \right).
\]
The third condition for $F$ to be an $M_n$-vector is that
\begin{enumerate}
\item[(V3)] for every cusp $s$ of $G$, the functions $f_k(q)$ are meromorphic at $q = 0$.
\end{enumerate}
When $G = \SL_2(\ZZ)$, the definition of an $M_n$-vector for $\SL_2(\ZZ)$ is a special case of the definition of a (meromorphic)  vvmf of weight $0$ as defined above.
\begin{ex}
  An important case to keep in mind is that studied by Eichler \cite{Eic1} and Shimura \cite{Shi1}.\ Let $f$ denote a scalar cusp-form of weight $k\geq 2$ on $\Gamma_0(N)$. Then the vector
\[
  F_f \df L_n(z)\left(\begin{matrix}
  0\\\vdots \\ 0\\ f(z)
\end{matrix}
  \right) = \left(\begin{matrix}
  z^{k-2}f(z)\\\vdots \\ zf(z)\\ f(z)
\end{matrix}
  \right)
\]
is an $M_{n}$-vector for $\SL_2(\ZZ)$, where $n = k-2$.\ Equivalently, $F_f$ defines a vvmf for $M_n$ of weight $0$.\ Note that in this case $\rho(T)$ is not diagonalizable and $F_f$ has a $\log q$-expansion (since $2\pi i z = \log q$).\ By integrating the real part of $F_f$ over parabolic cycles in $\uhp$, Eichler in the case $k = 2$, and Shimura in general, established an isomorphism $S_k(\Gamma_0(N)) \cong H^1(\Gamma_0(N),M_n)$ of real vector spaces.\ Here, $S_k(\Gamma_0(N))$ denotes complex vector space of cusp-forms for $\Gamma_0(N)$ of weight $k$ and $H^1(\Gamma_0(N),M_n)$ denotes the group cohomology of $\Gamma_0(N)$ with values in the representation $M_n$.
\end{ex}

More than forty years elapsed between the introduction of vvmfs into mathematics and the general study of vvmfs initiated by \cite{KnoMas2} and \cite{KnoMas3}.\ In \cite{Sel} the notion of vvmf had been confined to representations of $\Gamma$ with finite image.\ The work \cite{KnoMas2} expanded the definition of vvmf to cover all finite-dimensional representations of $\Gamma$, and the authors showed that Hecke's method for obtaining bounds on the growth of Fourier coefficients, when combined with results of Eichler \cite{Eic2}, can be applied to this broader class of vvmfs:
\begin{thm}
\label{t:fourierbound}
Let $(F, \rho)$ be a vector-valued modular form of weight $k$.\ There is a nonnegative constant 
$\alpha$, depending only on $\rho$, such that if $f_i(q) = \sum_{n \geq 0} a_n(i)q_N^n$ is the Fourier expansion of the $i$th coordinate of $F$, then $a_n(i) = O(n^{k+2\alpha})$.\ If $F$ is cuspidal then $a_n(i) = O(n^{\frac{k}{2}+\alpha})$.
\end{thm}
\begin{rmk}
Using an estimate of Eichler (loc.\ cit.) and the arguments in \cite{KnoMas2}, one can show that $\alpha \leq 0.345 \cdot\log\norm{\rho(S)}$ where
 for a matrix $A$ we define $\norm A\df \max_i \sum_j \abs{A_{ij}}$.\ If $\rho$ is unitarizable (e.g. if the image of $\rho$ is finite) one can take $\alpha$ to be zero.\ On the other hand, as we shall explain in Section \ref{SSMW}, there is no upper bound on $\alpha$ that applies uniformly for all $\rho$; one can find $\rho$ of any dimension $n$ for which the best-possible $\alpha$ is no less than $(n-2)/2$.
\end{rmk}

While the consistency of Theorem \ref{t:fourierbound} with the scalar case is appealing, when it first appeared it was not known that  nonzero vvmfs
exist for arbitrary representations.\ In order to address this, and to prove stronger structural results, the paper \cite{KnoMas3} restricts to studying vvmfs corresponding to representations $\rho$ such that $\rho(T)$ is of finite order\footnote{The paper \cite{KnoMas3} allows for a nontrivial multiplier, just as Selberg does in \cite{Sel}, but we will omit this from the discussion.}. 

Assume now that $\rho(T)$ is of finite order.\ Let $\cH_k(\rho)$ denote the space of holomorphic forms of weight $k$ associated to such a representation $\rho$ and let $\cS_k(\rho)$ denote the subspace of cusp-forms.\ Among other things, in \cite{KnoMas3} the authors show that $\cH_k(\rho)$ is finite-dimensional, and that $\cH_k(\rho) = 0$ if $k < -2\alpha$ where $\alpha$ is as in Theorem \ref{t:fourierbound}.\ The essential ingredients are the use of vector-valued Poincar\'e series associated to $\rho$ and a vector-valued version of the Petersson pairing.\ The authors show that for weights $k>2+2\alpha$, the Poincare series span $\cS_k(\rho)$, and that the difference between $\cS_k(\rho)$ and $\cH_k(\rho)$ is accounted for by suitably defined vector-valued Eisenstein series.\ Furthermore, for the same range of $k$, Theorem \ref{t:fourierbound} can be improved to obtain the classical estimate $a_n(i)=O(n^{k-1})$ for Fourier coefficients in the holomorphic case.\ Thus for large enough weights, the gross structure of (nonlogarithmic) holomorphic vector-valued modular forms is completely parallel to the classical case.

\begin{rmk}
 Recent work of Gannon \cite{Gan} has expanded the scope of many of these results to cover a far broader class of finite dimensional representations of $\SL_2(\ZZ)$. The essential idea is the use of the solution to the classical Riemann-Hilbert problem to obtain the existence of vvmfs.
\end{rmk}

\subsection{Vector-valued modular forms and the Riemann-Hilbert problem}
Let 
\begin{equation}
\label{eq:ode}
  L(f) = \frac{df^n}{dz^n} + a_1(z) \frac{df^{n-1}}{dz^{n-1}} + \cdots + a_n(z)f = 0
\end{equation}
denote a differential equation on $\PP^1(\CC)$, where each $a_i(z) \in \CC(z)$ is a rational function.\ Assume that the coefficients $a_i(z)$ share no common factor $(z-a)$ for $a \in \CC$.\ Let $S = \{s_1,\ldots, s_{t+1}\} \subseteq \PP^1(\CC)$ denote the finite set of singularities appearing in the coefficients $a_i(z)$.
\begin{dfn}
  A point $a \in \CC$ is said to be a \emph{regular singularity} of equation (\ref{eq:ode}) if the coefficient $a_i(z)$ has a pole of order at most $i$ at $a$ for $i = 1,\ldots, n$.\ The point $\infty \in \PP^1(\CC)$ is said to be a \emph{regular singularity} of (\ref{eq:ode}) if, after writing the equation in terms of $w = 1/z$, the result has $0$ as a regular singularity.\ The equation (\ref{eq:ode}) is said to be \emph{Fuchsian} if it has all points of $\PP^1(\CC)$ as regular singularities.
\end{dfn}
Note that (\ref{eq:ode}) is trivially regular singular at all points of $X = \PP^1(\CC) - S$. Assume that $t \geq 2$ and that $s_{t+1} = \infty$.\ The universal covering space of $X$ may be identified with the upper half-plane $\uhp$.\ Let $\pi \colon \uhp \to X$ denote the projection.\ Let $\pi_1(X,a)$ denote the topological fundamental group of $X$ with fixed basepoint $a \in X$, and let $G$ denote the (faithful) image of $\pi_1(X,a)$ in the automorphism group of $\uhp$, so that $X$ is conformally equivalent with $\uhp/G$.

Since (\ref{eq:ode}) has no singularity at $z = a$, Cauchy's thereom yields a local basis $F = (f_j)$ of solutions to (\ref{eq:ode}). Let $V = \CC^n$ be the $\CC$-span of the coordinates of $F$. By analytic continuation, the elements of $V$ may be regarded as holomorphic functions on $\uhp$.\ The analytic continuation of elements of $V$ about loops in $X$ defines the \emph{monodromy representation} $\rho \colon \pi_1(X,a) \to \GL_n(\CC)$.\ The vector of solutions $F$ satisfies the identity
\[
  F(\gamma z ) = \rho(\gamma) F(z),\quad\quad z \in \uhp,~ \gamma \in \pi_1(X,a),
\]
by definition of the monodromy representation.\ This resembles the definition of a vvmf of weight zero, but where $\Gamma$ is replaced by the fundamental group of $X$. Poincar\'e called such functions \emph{zetafuchsian systems} (cf. \cite{Poi} or \cite{SinTre})\footnote{More recently Stiller \cite{Sti} and others (e.g. \cite{Min}) have studied generalized automorphic forms defined relative to quite general monodromy representations}.

This paper focuses on the case when $S = \{0,1,\infty\}$.\ Let $\cE \subseteq \uhp$ denote the set of elliptic points for $\Gamma$.\ The automorphic function $K = 1728/j$ for $\Gamma$ defines a covering map of $\uhp-\cE$ onto $X = \PP^1 - \{0,1,\infty\}$.\ As a meromorphic function on the extended upper half plane, $K$ maps the cusps to $0$, the orbit $\Gamma i$ to $1$, and the orbit $\Gamma \rho$, where $\rho = e^{2\pi i/3}$, to $\infty$.\ The universal property of the universal covering space yields a commutative diagram
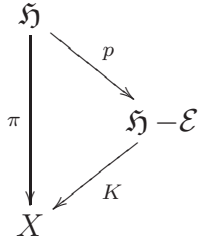
\[\xymatrix{
\uhp \ar[dd]_\pi\ar[dr]^p&\\
&\uhp - \cE\ar[dl]^K\\
X
}\]
of covering spaces.\ Let $\gamma_1$ and $\gamma_\infty$ in $\pi_1(X,a)$ denote the generating loops passing around $1$ and $\infty$, respectively.\ The loops $\gamma_1^2$ and $\gamma_\infty^3$ lift to closed loops about some elliptic point in $\uhp - \cE$, and they thus act trivially on this cover.\ The quotient $\pi_1(X,a)/\langle \gamma_1^2,\gamma_\infty^3\rangle$ is isomorphic with $\bar\Gamma$, and one sees that zetafuchsian systems whose monodromy factors through the diagram above give rise to vector-valued modular forms of weight $0$ (although possibly with bad behaviour at cusps and elliptic points).

This discussion proceeded from a differential equation and produced a monodromy representation.\
Suppose that we begin instead with a representation of $\bar\Gamma$, or more generally of the fundamental group of projective space deprived of some points.\ The question of realizing representations as the monodromy of a Fuchsian differential equation is the \emph{Riemann-Hilbert problem} (Hilbert's twenty-first problem).\ If one allows matrix differential equations with arbitrary singularities, then the fact that all such representations can be so realized follows from the fact that every holomorphic vector bundle on a noncompact Riemann surface is trivial.\ With a little more care one can show that every representation can be realized by a Fuchsian matrix differential equation (Theorem 31.5, \cite{For}).\ Given such a matrix differential equation, one may choose a cyclic vector to obtain a differential equation of the form (\ref{eq:ode}) with the same monodromy.\ However, even if one begins with a Fuchsian matrix system on $X$, the process of choosing a cyclic vector may introduce new \emph{apparent} singularities.\ These are singularities of the coefficients of the equation that are not singularities of the solutions to the equation about the ostensibly singular points.\ Plemlj \cite{Ple} showed that one can always realize a given representation as the monodromy of a differential equation (\ref{eq:ode}) regular on $X$ and with all but at most one of the points of $S$ as regular singularities of the equation.\ Plemlj claimed that the last possible singularity could be made regular, but Bolibrukh \cite{Bol1} gave a counterexample to this.\ Bolibrukh \cite{Bol2} and Kostov \cite{Kos} proved independently that irreducible representations can always be realized as the monodromy of a Fuchsian equation.

The works \cite{BanGan} and \cite{Gan} take the perspective of realizing vvmfs of weight $0$ as solutions to Fuchsian matrix differential equations on $\PP^1 - \{0,1,\infty\}$.\ In this way Gannon \cite{Gan} has used the solution to the Riemann-Hilbert problem to prove the existence, and extend much of the theory, of vector-valued modular forms to arbitrary representations of $\Gamma$.\ Additionally, the work \cite{Gan} further extends the theory to handle representations of other genus-zero Fuchsian groups.

In the final two sections of this paper we take an opposite perspective: we will show how a structural result for vvmfs (Theorem \ref{t:freemodule} on page 10) can be used to easily solve the Riemann-Hilbert problem for representations of $\Gamma$ of low degree.\ Several interesting examples involving generalized hypergeometric series will be discussed.\ Before turning to this, however, we briefly recall some facts about ordinary differential equations.

\begin{rmk}
  Several authors (e.g. \cite{PutUlm}, \cite{CreHaj1}, \cite{CreHaj2}) have studied the problem of computing differential equations with prescribed monodromy.
\end{rmk}

\subsection{The Frobenius method for solving Fuchsian equations}
\label{ss:frobenius}
The classical Frobenius method for solving linear equations allows one to find solutions to (\ref{eq:ode}) about regular singular points. For simplicity assume that $z = 0$ is a regular singular point of (\ref{eq:ode}) and proceed as follows: we seek a meromorphic and possibly multivalued solution of the form
\[
  f(z) = \sum_{m \geq 0}a_mz^{m+r}
\]
where $a_0 \neq 0$ and $r$ is some parameter for which we will solve. Since $0$ is a regular singularity of (\ref{eq:ode}), in a neighbourhood of $0$ we may write (\ref{eq:ode}) in the form
\[
  z^n\frac{df^n}{dz^n} + \left(\sum_{m \geq 0} b_{m,1}z^m\right)z^{n-1}\frac{df^{n-1}}{dz^{n-1}} + \cdots +\left(\sum_{m \geq 0} b_{m,n}z^m\right)f = 0,
\]
where $b_{0,j} \neq 0$ for $j = 1,\ldots, n$. Substituting $f(z)$ into this equation and considering the coefficient of $z^r$ in the resulting expression produces the equation:
\[
\frac{(r+n)!}{r!} + b_{0,1}\frac{(r+n-1)!}{r!} + b_{0,2}\frac{(r+n-2)!}{r!}+  \cdots + b_{0,n} = 0.
\]
The left side of this expression is in fact a polynomial in $r$, called the \emph{indicial polynomial} of (\ref{eq:ode}) near $z = 0$.\ If all of the roots $r_1,\ldots, r_n$ of the indicial polynomial are distinct, and no two differ by an integer, then it is possible to solve for $n$ distinct (though possibly multivalued) solutions to (\ref{eq:ode}) near $z = 0$ of the form $f_j(z) = z^{r_j}g_j(z)$, where $g_j(z)$ is holomorphic and single-valued near $0$.\ The local monodromy around $z = 0$ acts on  $f_j$ by mutiplication by $e^{2 \pi i r_j}$.

The indicial polynomial takes a simpler form if one expresses (\ref{eq:ode}) in terms of the derivation $\theta_z = z\frac{d}{dz}$. If
\[
  \theta_z^nf + g_1\theta^{n-1}_zf + \cdots g_nf = 0
\]
denotes (\ref{eq:ode}) in terms of this new derivation, then regularity at $z = 0$ is equivalent to having $g_j \in \pseries{\CC}{z}$ for all $j$. The indicial polynomial of (\ref{eq:ode}) near $z = 0$ is then simply $r^n + g_1(0)r^{n-1} + \cdots + g_n(0)$. 

The monodromy around $z = 0$ of the differential equations that arise below will be finite, and thus the indicial roots $r_j$ will be rational numbers.\ While the local monodromy only depends on the indicial roots $r_j$ mod $\ZZ$, in general one cannot always assume that the indicial roots lie between $0$ and $1$.\ In the examples below, however, it will be the case that the indicial roots lie between $0$ and $1$.

\section{The free-module theorem and modular linear differential equations}
\label{s:mldes}
\subsection{The free-module theorem}
\label{ss:freemod}
For each even integer $k \geq 4$, let $E_k$ denote the holomorphic Eisenstein series for $\Gamma$ of weight $k$ with the $q$-expansion
\[
  E_k(q) = 1 + \frac{2}{\zeta(1-k)}\sum_{n \geq 1} \sigma_{k-1}(n)q^n.
\]
The quasi-modular Eisenstein series $E_2$ is defined as a $q$-series in the same way. While $E_2$ is not a genuine modular form, its appearance in the modular derivative
\[
  D_k := q\frac{d}{dq} - \frac{kE_2}{12}\quad (k \in \ZZ)
\]
ensures that $D_k$ preserves modular forms (while increasing the weight by $2$). Indeed, for holomorphic functions in $\uhp$ we have
\begin{eqnarray}\label{eq:Dcommstroke}
(D_kf)|_{k+2}\gamma = D(f|_k \gamma)\ \ (\gamma \in \Gamma).
\end{eqnarray}

 Let $D_k^{(0)}$ denote the identity map on functions, and for $n \geq 1$ define
\[
  D_k^{(n)} := D_{k+2(n-1)}\circ D_k^{n-1}=D_{k+2(n-1)} \circ D_{k+2(n-2)} \circ \cdots \circ D_k.
\]
When there is no possibility for confusion, we will abuse notation and write $D = D_k$ and $D^n = D_k^{(n)}$. Ramanujan observed that
\begin{equation}
\label{eq:eisderivs}
D_2(E_2) = -\frac{E_4}{12},\quad D_4(E_4) = -\frac{E_6}{3},\quad\quad D_6(E_6) = -\frac{E_4^2}{2}.
\end{equation}

Because $D$ is a graded derivation of weight $2$ of the graded ring of modular forms for $\Gamma$, we may construct the graded (associative, noncommutative)  \emph{skew-polynomial ring} $R = \CC[E_4,E_6]\langle D\rangle$ formally obtained by adjoining $D$ to $\CC[E_4, E_6]$. Elements of $R$ are formal polynomials $\sum_i f_iD^i$
 with coefficients in $\CC[E_4, E_6]$. They are multiplied using the identity $Df-fD:= D_k(f)$, where $f$ is a scalar form of weight $k$. Ramanujan's identities (\ref{eq:eisderivs}) translate to the following commutation rules in $R$:
\begin{equation}\label{eq:commrules}
  [D,E_4] = -\frac{E_6}{3},\quad [D,E_6] = -\frac{E_4^2}{2}.
\end{equation}

If $\rho$ is a finite-dimensional representation of $\Gamma$, let $\cH(\rho) = \bigoplus_k \cH_k(\rho)$ denote the direct sum of the spaces of vvmfs for $\rho$ of varying weight. Coordinatewise multiplication makes $\cH(\rho)$ a graded $R$-module. In \cite{MarMas} one finds the following result about the structure of this module.
\begin{thm}[The free-module theorem]
\label{t:freemodule}
Suppose that $\rho$ is a complex representation of $\Gamma$ of dimension $n$ such that $\rho(T)$ is conjugate to a unitary matrix.\ Then $\cH(\rho)$ is a free $\CC[E_4,E_6]$-module of rank $n$.
\end{thm}
 Section \ref{s:hypergeometric} explains how this theorem can be used to manufacture differential equations satisfied by modular forms.\ While it has been known at least since the publication of \cite{Sti} that modular forms satisfy differential equations\footnote{\cite{Yan1} claims that this result dates back to the end of the nineteenth century, which is probably a reference to the pioneering work of Poincar\'e.}, in many cases the free-module theorem allows one to get one's hands on the differential equations in question quite easily.\ We will give several examples below.

\begin{rmk}
Several authors (e.g. \cite{Sti}, \cite{Yan1}, \cite{Yan2}) have discussed the problem of computing differential equations satisfied by scalar modular forms.
\end{rmk}
\begin{rmk}
The proof of Theorem \ref{t:freemodule} in \cite{MarMas} formally resembles the following result in the theory of $D$-modules: let $W_n = \CC[x_1,\ldots, x_n]\langle \partial_1,\ldots, \partial_n\rangle$ denote the \emph{Weyl algebra} of degree $n$.\ One thinks of this as a subalgebra of the algebra of linear endomorphisms of $\CC[x_1,\ldots, x_n]$ where $x_j$ acts like multiplication by $x_i$ and $\partial_i$ acts like $d/dx_i$.\ One can show that if $M$ is an $W_n$-module that is finitely generated as an $\CC[x_1,\ldots, x_n]$-module, then $M$ is free of finite rank as an $\CC[x_1,\ldots, x_n]$-module.\ Analogously, one of the key steps in the proof of the free-module theorem given in \cite{MarMas} is to show that $\cH(\rho)$ is finitely generated over $\CC[E_4,E_6]$. 
\end{rmk}

The free-module theorem tells us that the \emph{Hilbert-Poincar\'{e}} series of $\mathcal{H}(\rho)$,
i.e., the formal series $\sum_k \dim_\CC\mathcal{H}_k(\rho)t^k$, is a \emph{rational function}.\ More precisely, because 
$\mathcal{H}(\rho)$ has rank $n$ and the Hilbert-Poincar\'{e} series of $\CC[E_4, E_6]$ is $1/(1-t^4)(1-t^6)$, 
there are integers $k_1\leq ... \leq k_n$ such that
\[
\sum_k \dim\mathcal{H}_k(\rho)t^k = \frac{t^{k_1}+\cdots+t^{k_n}}{(1-t^4)(1-t^6)}.
\]
This is a sort of implicit Riemann-Roch theorem inasmuch as it tells us how to compute $\dim\mathcal{H}_k(\rho)$.\ The integers $k_1, \ldots, k_n$ are uniquely determined by $\rho$, being the weights of a free basis of $\mathcal{H}(\rho)$.\ They satisfy some additional constraints \cite{BanGan}, and can be computed in low-dimensional cases (e.g. \cite{Mas1}, \cite{Mar}, \cite{Mar2}), but in general they remain mysterious.

\subsection{Modular Wronskian}
\label{SSMW}
Let $\rho$ be an $n$-dimensional representation of $\Gamma$ with $\rho(T)$ unitarizable, so that the results of \cite{KnoMas2}, \cite{KnoMas3} and the free module theorem all apply to $\cH(\rho)$.\ The paper \cite{Mas3} is concerned with the problem of determining best-possible bounds $\beta$ such that for all $k < \beta$, the space $\cH_k(\rho)$ of weight $k$ vvmfs for $\rho$ is trivial. The key idea in that work is the use of the \emph{modular Wronskian} of a vvmf. If $F =\ ^t(f_1, ..., f_n)$ is a vvmf we define the modular Wronskian of $F$ as
\[
  W(F) = \det(F, DF, D^2F,\ldots, D^{n-1}F).
\]

The modular Wronskian is a very useful gadget inasmuch as it is both a scalar holomorphic form of weight $n(n+k-1)$ and ($\pm$ a power of $q$) times the usual Wronskian of $f_1, ..., f_n$.\
With a suitable normalization, one sees (cf.\ Theorem 3.7 of \cite{Mas3}) that $W(F) = \eta^{24\lambda}G$ where $G$ is a scalar holomorphic modular form of weight $n(n+k-1)-12\lambda$, where $\lambda$ is the sum of the exponents of the leading powers of $q$ of the components of $F$.\ This implies that $\cH_k(\rho) = 0$ if $k \leq 1-n,$ with equality only if $n=1$ and $\rho$ is the trivial representation of $\Gamma$.\ The modular Wronskian can be used to good effect in the study of vvmfs and is a useful tool in its own right.\ But it really presages the connections between vvmfs and linear differential equations that we take up in the next Subsection.

As a first application of ideas borrowed from the Frobenius-Fuchs theory, Theorem 4.3 of \cite{Mas3}  shows that a vvmf $F$ of weight $k$ satisfies the differential equation $W(F) = c\eta^{24\lambda}$ for a constant $c$ if, and only if, the components of $F$ form a fundamental system of solutions of an MLDE of weight $k$ (see Definition \ref{d:MLDE} below).\ This result, a modular version of Abel's theorem on the nonvanishing of Wronskians, leads to the existence of vvmfs $(F, \rho)$ with $\dim\rho=n\geq 2$ and $\mathcal{H}_{2-n}(\rho)\not= 0$.\ Thus the best-possible
value of $\beta$ is $2-\dim\rho$ for $n\geq 2$ (cf.\ the discussion following Theorem \ref{t:fourierbound}).

\subsection{Modular linear differential equations}
\label{SSMLDE}

The definition of an MLDE is as follows:
\begin{dfn}
\label{d:MLDE}
The general \emph{modular linear differential equation} (or MLDE) of weight $(k,l)$ and degree $n$ is the differential equation defined by the following modular differential operator
\begin{equation}\label{eq:DO}
 L_n = L_n^{(k,l)} = \sum_{j = 0}^n P_{k + 2(n-j)}(E_4,E_6)D_l^{(j)},
\end{equation}
where $P_d(E_4,E_6)$ denotes a homogeneous polynomial in $E_4$ and $E_6$ of degree $d$, where $E_4$ is of degree $4$ and $E_6$ is of degree $6$.\ If the MLDE has weight $(0, l)$, it is called 
\emph{monic} if the leading coefficient $P_0$ satisfies $P_0=1$.
\end{dfn}

For example, the general MLDE of weight $(8,l)$ and degree $3$ is defined by the differential operator
\[
  aE_4^2D^3_l + bE_4E_6D^2_l + (cE_4^3 + dE_6^2)D_l + eE_4^2E_6
\]
for complex parameters $a$ through $e$. A modular differential operator of weight $(k,l)$ and degree $n$ maps modular forms of weight $l$ to modular forms of weight $l + k + 2n$. In the low-dimensional examples below we will be mostly interested in the case $k = l = 0$.

The modular invariance of $D$ implies that the solution space of the MLDE defined by $L_n$ is a right $\Gamma$-module with respect to the stroke operator.\ A fundamental system of solutions of the MLDE are therefore the components of a vvmf (possibly meromorphic at infinity, and possibly logarithmic), as was explained in 
Section \ref{SSDandP}.\ In fact the converse is also true.\ Indeed, the results of \cite{Mas3} lead the second author to the study of the differential-algebraic structure of $\cH(\rho)$ for two-dimensional $\rho$ in \cite{Mas1}, and this brought Marks-Mason to the proof of the free-module theorem in \cite{MarMas}.\ 

The free-module theorem can be used to find MLDEs satisfied by the coordinates of vvmfs.\ The general idea is as follows: let $(F, \rho)$ be a (nonzero) vvmf of weight $l$ and with $\dim\rho=n$.\
Since $\mathcal{H}(\rho)$ has rank $n$ as a $\CC[E_4, E_6]$-module, the vvmfs $F, DF, ..., D^nF$ must be linearly dependent.\ That is we can find a weight $k$ and scalar forms $P_{k+2(n-j)}(E_4, E_6)$ so that $\sum_{j=0}^n P_{k+2(n-j)}(E_4, E_6)D^{j}_lF=0$.\ Thus the component functions $f_1, ..., f_n$ of
$F$ are solutions of an MLDE of weight $(k, l)$.\ Furthermore, if the $\{f_i\}$ are linearly independent (which may usually
be assumed in practice, and which in any case holds if $\rho$ is irreducible) then they are a \emph{basis} for the solution space of the MLDE.

Suppose that the coordinates $f_i$ of a vvmf $F$ are linearly independent.\ Were we to be satisfied with \emph{any} order $n$ MLDE satisfied by the $f_i$, we could employ an old device of Fuchs. Indeed, 
\[
\det\left(\begin{array}{cccc} f & Df & \hdots & D^nf \\ f_1 & Df_1 & \hdots & D^nf_1 \\ \vdots & \vdots &  & \vdots \\ f_n & Df_n & \hdots & D^nf_n\end{array}\right)=0
\]
is an MLDE such that the leading coefficient is the modular Wronskian $W(F)$.\ Often this is not good enough, however --- in practice MLDEs of low weight are easier to solve, and the leading coefficient $W(F)$ of this equation typically has large weight.\ The free module theorem allows us to find more amenable differential equations in many cases of interest.\ We will discuss the cases of representations of dimension $2$ and $3$ in detail in Section \ref{s:hypergeometric} below.\ Before turning to this we must discuss some generalities on reparameterizing MLDEs in terms of Hauptmoduls, with a view towards solving the various differential equations that will arise.

\subsection{Modular reparameterization of MLDEs}
\label{ss:modreparam}
Let
\begin{align*}
  \Delta &= (E_4^3-E_6^2)/1728, & j &= E_4^3/\Delta,\\
  A &= E_6/E_4, & K &= 1728/j.
\end{align*}
Note that $K$ is a local parameter at infinity on the $j$-line.\ Our goal is to reexpress certain MLDEs in terms of this parameter.\ To this end we introduce the notation $\theta = q\frac{d}{dq}$ and $\theta_K = K\frac{d}{dK}$.\ The following identities, which follow easily from (\ref{eq:eisderivs}) of Section \ref{ss:freemod}, will be used frequently below:
\begin{eqnarray}
   \theta &=& A\theta_K,\quad \theta(A) = \frac{E_2E_4E_6 - 3E_4^3 + 2E_6^2}{6E_4^2}, \notag \\
  E_4 &=& \frac{A^2}{1-K},\quad E_6 = \frac{A^3}{1-K}. \label{Aids}
\end{eqnarray}

These identities allow us to reexpress MLDEs in terms of $\theta_K$ and the local parameter $K$.\ For monic MLDEs of degree $\leq 3$, the resulting equations turn out to be generalized hypergeometric in the sense of Beukers-Heckmann \cite{BeuHec} (cf.\ Example \ref{exdeg23} below).\ Although this pattern fails to hold in higher degrees, in the monic case one obtains Fuchsian systems with regular singularities only at $\{0, 1, \infty\}$.\ To explain this, consider a modular linear differential operator of weight $(2k,0)$ and degree $n$:
\[
\sum_{j = 0}^n P_{2(k + n-j)}(E_4,E_6)D^{j},
\]
where
\[
  P_{2(k + n-j)}(E_4,E_6) = \sum_{2a+3b = k+n-j}c_{ab}^j E_4^aE_6^b = A^{k+n-j}\left(\sum_{2a+3b = k+n-j}\frac{c_{ab}^j}{(1-K)^{a+b}}\right).
\]
The following lemma expresses $D^j$ in terms of $\theta_K$.
\begin{lem}
\label{l:dtotheta}
For all $j \geq 1$ one has $D_0^{(j)} = A^jP_j(K,\theta_K)$ where\[
  P_j(K,\theta_K) = \sum_{r = 1}^j \frac{p_{jr}(K)}{(1-K)^{\rho(j,r)}}\theta_K^r,
\]
where $\rho(j,r) = \inf \{j-r,\floor{j/2}\}$ and each $p_{jr}(K) \in \CC[K]$ is a polynomial of degree $\leq \rho(j,r)$.
\end{lem}
\begin{proof}
For $j = 1$ the lemma follows since $D = A\theta_K$.\ For $j > 1$ we proceed by induction, noting first that $D(A) = -A^2\frac{1+2K}{6(1-K)}$.\ Thus,
\begin{align*}
  D^{j+1} &= D\left(A^jP_j(K,\theta_K)\right)\\
&=-jA^{j+1}\left(\frac{1+2K}{6(1-K)}\right)P_j(K,\theta_K) + A^{j}D_0\left(P_j(K,\theta_K)\right)\\
&= A^{j+1}\left(\theta_K(P_j(K,\theta_K)) - j\left(\frac{1+2K}{6(1-K)}\right)P_j(K,\theta_K) \right),
\end{align*}
and we deduce that
\[
  P_{j+1}(K,\theta_K) = \theta_K(P_j(K,\theta_K)) - j\left(\frac{1+2K}{6(1-K)}\right)P_j(K,\theta_K).
\]
Write $P_j(K,\theta_K) = \sum_{i = 1}^j p_{ij}(K)\theta^i$ for $p_{ij}(K) \in \CC(K)$.\ Then one deduces the following recurrence formulae for the coefficients $p_{ij}(K)$:
\begin{align*}
  p_{(j+1)(j+1)} &= 1,\\
  p_{i,(j+1)} &= \theta_K(p_{ij}) - j\frac{(1+2K)}{6(1-K)}p_{ij} + p_{(i-1)j},\\
p_{1(j+1)} &= \theta_K(p_{1j}) - j\frac{(1+2K)}{6(1-K)}p_{1j},
\end{align*}
for $i$ between $2$ and $j$, while $P_1(K,\theta_K) = \theta_K$.\ The first identity above proves the claim for $p_{jj}(K)$.\ If $j$ is odd then the second two recursive identities show that the numerator of $p_{1(j+1)}$ is of degree at most one larger than that of $p_{1j}$, and the degree of $(1-K)$ in the denominator of $p_{1(j+1)}$ is also at most one larger.\ Thus, we are reduced to proving that the degrees do not increase in the case when $j$ is even and $i \leq j/2$.

Momentarily write $p_j = p_{1j}$ and let $j =2k$.\ By induction we may write $p_{j} = \alpha/(1-K)^{k}$ where $\alpha \in \CC[K]$ of degree at most $k$.\ We see that
\begin{align*}
 p_{j+1} &= \theta_K\left(\frac{\alpha}{(1-K)^k}\right) - 2k\frac{1+2K}{6-6K}\frac{\alpha}{(1-K)^k}\\
&= \frac{\theta_K(\alpha)(1-K)^k + k\alpha(1-K)^{k-1}K}{(1-K)^{2k}} - k\frac{(1+2K)\alpha}{3(1-K)^{k+1}}\\
&= \frac{3\theta_K(\alpha)(1-K) + 3k\alpha K-k(1+2K)\alpha}{3(1-K)^{k+1}}\\
&= \frac{3\theta_K(\alpha) - k\alpha}{3(1-K)^{k}}.
\end{align*}
Since $\theta_K(\alpha)$ is a polynomial in $K$ of the same degree as $\alpha$, this proves the claim for $p_{1(j+1)}$.\ The proof for terms $p_{i(j+1)}$ where $i$ is between $2$ and $j/2$ is identical, as the additional term $p_{(i-1)j}$ appearing in the recursive formula is already of the desired form $\beta/(1-K)^{j/2}$ by induction.
\end{proof}

\begin{rmk}
It does not appear to be easy to write down a general formula for the rational functions $p_{j,j+a} \in \CC(K)$ arising in the previous proof.
\end{rmk}

From Lemma \ref{l:dtotheta} one  deduces the following result.
\begin{thm}
\label{t:monicfuchsian}
  Every monic MLDE in terms of $D_0$ becomes a Fuchsian differential equation on $\PP^1\setminus{\{0,1,\infty\}}$ when expressed in terms of the parameter $K = 1728/j$.\ More precisely, under this reparameterization and after dividing by $A^n$, a general degree $n$ monic MLDE in terms of $D_0$ takes the form
\[
  \left(\sum_{r = 0}^{\floor{n/2}-1}\frac{f_{n-r}(K)}{(1-K)^r}\theta_K^{n-r} + \frac{1}{(1-K)^{\floor{n/2}}}\sum_{r = \floor{n/2}}^{n}f_{n-r}(K)\theta_K^{n-r}\right)f = 0,
\]
where $\theta_K = K\frac{d}{dK}$, $f_n(K) = 1$, and $f_{n-r}(K)$ is a polynomial of degree $\leq \inf\{r,\floor{n/2}\}$.
\end{thm}
\begin{proof}
The cases of degree $\leq 1$ are trivial.\ The general monic modular linear differential operator of degree $n \geq 2$ is of the form
\[
  D^n + \sum_{k = 0}^{n-2}\sum_{2a+3b = 2(n-k)} c_{ab}E_4^aE_6^bD^k \quad (c_{ab}\in \CC).
\]
By the preceding lemma we may reexpress this as
\[
  A^n\sum_{r = 1}^n\frac{p_{nr}(K)}{(1-K)^{\rho(n,r)}}\theta_K^r + \sum_{k = 0}^{n-2}A^{n-k}\sum_{2a+3b = n-k} \frac{c_{ab}}{(1-K)^{a+b}}\left(A^k\sum_{r = 1}^k\frac{p_{kr}(K)}{(1-K)^{\rho(k,r)}}\theta_K^r\right)
\]
which, after collecting terms and rescaling by $A^n$, yields the operator:
\begin{align*}
&\theta^n_K - \frac{n(n-1)}{12}\left(\frac{1+2K}{1-K}\right) \theta_K^{n-1}+\sum_{r = 1}^{n-2}\left(\frac{p_{nr}(K)}{(1-K)^{\rho(n,r)}} + \sum_{k = r}^{n-2}\sum_{2a+3b = n-k} \frac{c_{ab}p_{kr}(K)}{(1-K)^{a+b+\rho(k,r)}}\right)\theta_K^r + \\
&\quad\quad +\sum_{2a+3b = n} \frac{c_{ab}}{(1-K)^{a+b}}.
\end{align*}
Since $\rho(k,r) =  \inf\{k-r,\lfloor k/2\rfloor\}$ and $a+b < (n-k)/2$, the theorem follows.
\end{proof}

\begin{ex}
\label{exdeg23}
  The degree $2$ and $3$ equations are generalized hypergeometric equations, in the sense of \cite{BeuHec}, defined by the following differential operators.
\begin{align*}
\theta_K^{2}&  - \left(\frac{2K+1}{6(1-K)}\right) \theta_K + \frac{a}{1-K},\\
\theta_K^{3}& - \left(\frac{2K+1}{2(1-K)}\right) \theta_K^{2}  + \left(\frac{18a+1-4K}{18(1-K)} \right) \theta_K + \frac{b}{1-K}.
\end{align*}
\end{ex}

\begin{ex}
\label{exdeg45}
  The degree $4$ and $5$ equations are the differential equations defined by the following differential operators.
\begin{align*}
\theta_K^{4} - &\left(\frac{2K+1}{1-K}\right) \theta_K^{3} + \left(\frac{44 K^{2} - 4(9 a+7)K + 36 a + 11}{36 (1-K)^2}\right) \theta_K^{2} + \\
&\quad \left(\frac{8 K^{2} - 4(3 a + 9 b + 1)K - 6 a + 36 b - 1}{36 (1-K)^2}\right) \theta_K + \frac{c}{(1-K)^2},\\
\theta_K^{5} - &\left(\frac{5(2K+1)}{3(1-K)}
\right) \theta_K^{4} + \left(\frac{140 K^{2} - 4(9 a + 10) K + 36 a + 35}{36(1-K)^2}\right)\theta_K^{3} + \\
&\quad\left(\frac{200 K^{2} - 4(27 a + 27 b + 10) K - 54 a + 108 b - 25}{108(1-K)^2}\right)\theta_K^{2} + \\
&\quad\left(\frac{16 K^{2} - 2(6 a + 3 b + 1)K + 3 a - 9 b + 54 c + 1}{54(1-K)^2}\right)\theta_K + \frac{d}{(1-K)^2}.
\end{align*}
\end{ex}

\begin{rmk}
While the equations in Examples \ref{exdeg23} and \ref{exdeg45} are rigid in the sense of Katz \cite{Kat}, one does not always obtain rigid equations from monic MLDEs in degrees $6$ and higher.
\end{rmk}

It is not true that all MLDEs are pullbacks of Fuchsian equations in this manner.\ The following theorem identifies exactly which ones \emph{are}.

\begin{thm}
\label{t:mainfuchsian}
An MLDE of degree $n$ in terms of $D = D_0$ is the pullback via $K$ of a Fuchsian equation on $\PP^1\setminus{\{0,1,\infty\}}$ if, and only if, when rescaled by $E_6^{2n}$ it takes the form
\begin{eqnarray}\label{basicfuchs}
  ((E_4E_6)^nD^n + F_1(E_4E_6)^{n-1}D^{n-1} + \cdots + F_{n-1}E_4E_6D + F_n)f = 0,
\end{eqnarray}
where each $F_k$ is a holomorphic modular form of weight $12k$.
\end{thm}
\begin{proof}
We first show that every such equation becomes Fuchsian on $\PP^1 \setminus{\{0,1,\infty\}}$. Write a generic modular operator as in (\ref{basicfuchs}) in the form
\begin{equation}
\label{eq:genfuchsMLDE}
  (E_4E_6)^nD^n + \sum_{i = 1}^n\sum_{j = 0}^i a_{ij}E_4^{n+3j-i}E_6^{n+i-2j}D^{n-i}\ \ (a_{ij}\in \CC).
\end{equation}
Lemma \ref{l:dtotheta} tells us that for all $j \geq 1$ we have 
\[
  D^j = A^j\sum_{r = 1}^j \frac{p_{jr}(K)}{(1-K)^{\rho(j,r)}}\theta_K^r,
\]
where $\rho(j,r) = \inf \{j-r,\lfloor j/2\rfloor\}$ and each $p_{jr}(K) \in \CC[K]$ is a polynomial of degree $\leq \rho(j,r)$. We may substitute this, along with the identities (\ref{Aids}), into (\ref{eq:genfuchsMLDE}) to obtain the differential operator
\begin{align*}
  \frac{A^{5n}}{(1-K)^{2n}}\left(A^n\sum_{r = 1}^n \frac{p_{nr}(K)}{(1-K)^{\rho(n,r)}}\theta_K^r\right)& + \sum_{i = 1}^{n-1}\sum_{j = 0}^i a_{ij}\frac{A^{5n + i}}{(1-K)^{2n+j}}\left(A^{n-i}\sum_{r = 1}^{n-i} \frac{p_{(n-i)r}(K)}{(1-K)^{\rho(n-i,r)}}\theta_K^r\right)\\
&+\sum_{j=0}^n a_{nj}\frac{A^{6n}}{(1-K)^{2n+j}}.
\end{align*}
Cancelling the factor of $A^{6n}(1-K)^{-2n}$ yields
\[
  \left(\sum_{r = 1}^n \frac{p_{nr}(K)}{(1-K)^{\rho(n,r)}}\theta_K^r\right) + \sum_{i = 1}^{n-1}\sum_{j = 0}^i a_{ij}\left(\sum_{r = 1}^{n-i} \frac{p_{(n-i)r}(K)}{(1-K)^{j+\rho(n-i,r)}}\theta_K^r\right) + \sum_{j = 0}^n\frac{a_{nj}}{(1-K)^j}.
\]
To see that this is Fuchsian, rewrite this as
\[
  \theta_K^n + \sum_{r = 1}^{n-1}\left(\sum_{i = 1}^{n-r}\sum_{j = 0}^{i}\frac{a_{ij}p_{(n-i)r}(K)}{(1-K)^{j+\rho(n-i,r)}}\right)\theta_K^r + \sum_{j = 0}^n\frac{a_{nj}}{(1-K)^j},
\]
and then note that $i + \rho(n-i,r) \leq n-r$ for all $i \in \{1,2,\ldots, n-r\}$.

Now we need to show that every Fuchsian equation on $\PP^1\setminus{\{0, 1, \infty\}}$ can be expressed as above.\ Begin by inverting the relation
\begin{eqnarray*}
D^n = A^n\left(\sum_{r=0}^{\floor{n/2}-1} \frac{p_{n,n-r}(K)}{(1-K)^{r}}\theta_K^{n-r}+
\frac{1}{(1-K)^{\floor{n/2}}} \sum_{r=\floor{n/2}}^{n-1} p_{n,n-r}(K)\theta_K^{n-r}\right)
\end{eqnarray*}
of Lemma \ref{l:dtotheta} expressing $D^n$ in terms of powers of $\theta_K$. Thus,
\begin{eqnarray*}
\left(\begin{array}{c}D^n \\ D^{n-1} \\ \vdots \\ D^2 \\ D^1\end{array}\right)=
\left(\begin{array}{ccccc}A^n & 0 &  & 0 & 0 \\0 & A^{n-1} &  & 0 & 0 \\\vdots & \vdots & \ddots & \vdots & \vdots \\0 & 0 &  & A^2 & 0 \\0 & 0 &  & 0 & A\end{array}\right) \left(\begin{array}{ccccc}1 & g_{01} &g_{02}  & \hdots &  g_{0(n-1)}\\  & 1 & g_{12} & \hdots &  g_{1(n-1)}\\  &  & 1 & \hdots &  g_{2(n-1)}\\ &  &  & \ddots &  \vdots\\ &  &  &  & 1\end{array}\right)   \left(\begin{array}{c}\theta_K^n \\ \theta_K^{n-1} \\ \vdots \\ \theta_K^2 \\ \theta_K\end{array}\right)
\end{eqnarray*}
\begin{eqnarray*}
g_{rs}= \mbox{coefficient of $\theta_K^{n-r-s}$ in $A^{-(n-r)}D^{n-r}$}=\frac{p_{(n-r)(n-r-s)}}{(1-K)^{\rho(n-r,n-r-s)}},
\end{eqnarray*}
where $0 \leq r \leq s \leq n-1$. Note that $\rho(n-r,n-r-s) = \inf\{s,\floor{(n-r)/2}\}$.

 Define elements $h_{ij} \in \CC[K][1/(1-K)]$ by inverting the relation above:
\begin{eqnarray*}
 \left(\begin{array}{c}\theta_K^n \\ \theta_K^{n-1} \\ \vdots \\ \theta_K^2 \\ \theta_K\end{array}\right)=
 \left(\begin{array}{ccccc}1 & h_{01} &h_{02}  & \hdots &  \\  & 1 & h_{12} & \hdots &  \\  &  & \ddots &  &  \\ & 0 &  & 1 &  \\ &  &  &  & 1\end{array}\right)
\left(\begin{array}{c}A^{-n}D^n \\ A^{-(n-1)}D^{n-1} \\ \vdots \\ A^{-2}D^2 \\ A^{-1}D\end{array}\right).  \end{eqnarray*}
Note that
\[\theta_K^n =\sum_{s=0}^{n-1} h_{0s}A^{-(n-s)}D^{n-s} = E_6^{-n}\left\{\sum_{s=0}^{n-1} h_{0s}E_6^sE_4^{n-s}D^{n-s}\right\},\]
and
\begin{eqnarray*}
h_{0s} = (-1)^s\det \left(\begin{array}{ccccc}g_{01} & g_{02} & \hdots & &g_{0s} \\ 1 & g_{12} & \hdots && g_{1s} \\0 & 1 & \ddots & &\vdots \\ \vdots & \vdots & \ddots & & 
\\0 & 0 & \hdots & 1&g_{(s-1)s}\end{array}\right).
\end{eqnarray*}

We claim that the denominator of $h_{0s}$ is ``no worse'' than $(1-K)^s=(E_6^2/E_4^3)^{s}$.\ To prove this, for each $0 \leq r < s$ let $\delta_r$ denote the determinant of the matrix appearing in the definition of $h_{0s}$, but with the first $r$ rows and columns deleted.\ We will show by descending induction that $\delta_r$ has denominator no worse than $(1-K)^{s-r}$.\ When $r = s-1$ and $\delta_{s-1} = g_{(s-1)s}$, this is clear. For $r < s-1$ we have
\[
\delta_r = g_{r(r+1)}\delta_{r+1}  - \delta'.
\]
By induction, the denominator for $\delta_{r+1}$ is no worse than $(1-K)^{s-r-1}$, while that for $g_{r(r+1)}$ is $1-K$.\ So the first term has denominator at worst $(1-K)^{s-r}$.\ On the other hand, $\delta'$ is the determinant of a matrix whose first row consists of entries $g_{r,j}$.\ These have denominator a power of $(1-K)$ that is at most one higher than that appearing in the $g_{r+1,j}$.\ Hence, the denominator of $\delta'$ contains at most one more factor of $(1-K)$ than in $\delta_{r+1}$.\ Hence $(1-K)^{s-r}\delta_r$ is a polynomial.\ If we take $r = 0$ then we deduce that $(1-K)^s\delta_0 = (-1)^s(1-K)^sh_{0s}$ is a polynomial, which proves the claim.

 Write 
 \begin{eqnarray*}
 h_{0s} = \frac{H_s}{(1-K)^s}=\frac{H_s E_4^{3s}}{E_6^{2s}}.
 \end{eqnarray*}
  By the previous paragraph we have
 $H_s\in \mathbf{C}[K]$, moreover $H_s$ has degree at most $s$.
It follows that for $0\leq j\leq i\leq n$ we have
\begin{eqnarray*}
E_6^{2n}\frac{\theta_K^{n-i}}{(1-K)^j} &=& E_6^{2n}\left(\frac{E_4^3}{E_6^2}\right)^j E_6^{i-n}
\sum_{s=0}^{n-i-1} \frac{H_s E_4^{3s}}{E_6^{2s}}  E_6^sE_4^{n-i-s}D^{n-i-s} \\
&=& E_6^{2(i-j)}\sum_{s=0}^{n-i-1} H_s
E_4^{3(s+j)} (E_4E_6)^{n-i-s}D^{n-i-s}.
\end{eqnarray*}

Observe that because $K= 1728\Delta/E_4^3$ and $H_s$ has degree at most $s$, 
the forms $H_sE_4^{3(s+j)}$ are in fact \emph{holomorphic} of weight $3(s+j)$.\ Consequently,  the operator
$E_6^{2n}\frac{\theta_K^{n-i}}{(1-K)^j}$ is of the correct shape, i.e., it takes the form of
the left-hand-side of (\ref{basicfuchs}) where $F_{i+s}= E_6^{2(i-j)}\sum_s H_sE_4^{3(s+j)}$
has weight $12(s+i)$.\ Finally, \emph{every} Fuchsian operator of degree $n$ on $\PP^1\setminus{\{0, 1, \infty\}}$ can be written as a sum 
\begin{eqnarray*}
\sum_{0\leq j \leq i \leq n} a_{ij} \frac{\theta_K^{n-i}}{(1-K)^j}\ \ (a_{ij}\in \mathbf{C}).
\end{eqnarray*}
Therefore, the general case follows from the special case of a single summand that we just established.
\end{proof}

\begin{rmk} Dividing throughout by $\Delta^n$, we can rewrite the left hand side of
(\ref{basicfuchs}) as a polynomial differential operator in the weight zero operator $(E_4E_6/\Delta)D$
with coefficients in the ring of weight zero modular functions with poles at $\infty$.\ This is the preferred approach of Bantay-Gannon \cite{BanGan}, in which everything happens at weight $0$.
\end{rmk}

\begin{rmk}
Theorem \ref{t:mainfuchsian} appears to be new.\ It is a generalization of Theorem A in \cite{Tsu}, which treats the case of equations of degree $2$.\ Note that the differential equations in \cite{Tsu} are expressed relative to the differential operator $q\frac{d}{dq}$ rather than $D$.
\end{rmk}

\section{Vector-valued modular forms and hypergeometric series}
\label{s:hypergeometric}
\subsection{The monic MLDE of degree $2$}\label{SSMLDE2}
Let $\rho$ be a two-dimensional representation of $\Gamma$.\ For simplicity assume that $\rho$ is irreducible and that $\rho(T)$ has finite order. Under these hypotheses, the free-module theorem applies and $\rho(T)$ has distinct eigenvalues (otherwise $\rho$ factors through the abelianization of $\Gamma$).\ Let $F,G$ denote a basis for $\cH(\rho)$ as an $\CC[E_4,E_6]$-module, where $k$ and $l$ are the weights of $F$ and $G$ respectively, with $k \leq l$.\ Since there are no holomorphic modular forms of weight $2$ for $\Gamma$, the free-module theorem implies that $DF = \alpha G$ for some complex scalar $\alpha$.\ As explained in Section \ref{SSMLDE}, the irreducibility of $\rho$ ensures that $DF\not= 0$.\ We deduce that $\cH(\rho)$ is the free $\CC[E_4,E_6]$-module spanned by $F$ and $DF$.\ Since $D^2F \in \cH_{k+4}(\rho)$, it must be a linear combination of $F$ and $DF$ with coefficients in $\CC[E_4, E_6]$.\ Since $DF$ has weight $k+2$ and there are no nonzero scalar forms of weight $2$ it follows that the coordinates of $F$ are a basis of solutions to an MLDE of the form
\[
  D^2_{k}f + aE_4f = 0,
\]
where $a \in \CC$.\ In this way we see that the components of a minimal weight vvmf for such a $\rho$ satisfy a \emph{monic} MLDE. 

 After reparameterizing this equation via $K$ as discussed in Section \ref{ss:modreparam} (cf. Example \ref{exdeg23}), when $k = 0$ one obtains the hypergeometric differential equation
\begin{equation}
\label{eq:hypgeom}
  \left(\theta_K^2 - \left(\frac{2K+1}{6(1-K)}\right)\theta_K + \frac{a}{1-K}\right)f = 0.
\end{equation}
In \cite{FraMas} it was observed that since $D(\eta) = 0$, one can always reduce to the case where the minimal weight is $0$ by rescaling $F$ by a power of the eta-function\footnote{This rescaling technique also appears in \cite{Sti}, where Stiller rescales by powers of a holomorphic modular form of weight $1$ to reduce to the weight zero case. For this reason \cite{Sti} focuses on subgroups of $\Gamma$ that do not contain the matrix $-I$, so that forms of weight one exist. Stiller remarks that this is inessential. For example, if one allows more general mutlipliers one could work with $\eta$ instead, as we do here.}. This has the effect of replacing $\rho$ by a twist $\rho\otimes \chi^{-k}$ where $\chi$ is the $1$-dimensional representation of $\Gamma$ satisfying
\begin{equation}\label{eq:chidef}
  \chi\twomat 0{-1}10 = i,\quad\quad \chi\twomat 1101 = e^{\pi i / 6}.
\end{equation}
This is the character corresponding to $\eta^2$, in the sense that $\eta^2(\gamma\tau)=\chi(\gamma)\eta^2(\tau)$ for $\gamma\in\Gamma$. Making the reduction to weight zero allows us to express the coordinates of $F$ in terms of $\eta$, the local parameter $K$, and hypergeometric series.

Let $a_0,a_1,\ldots, a_n$ and $b_1, \ldots ,b_{n}$ denote complex numbers such that no $b_i$ is a negative integer. The corresponding \emph{hypergeometric series} is defined by
\[
  _nF_{n-1}(a_0,\ldots, a_n; b_1,\ldots, b_{n};z) = 1 + \sum_{r\geq 1} \frac{(a_0)_r(a_1)_r\cdots (a_n)_r}{(b_1)_r(b_2)_r\cdots (b_{n})_r} \frac{z^r}{r!},
\]
where for $r \in \ZZ_{\geq 1}$ and $\alpha \in \CC$,  we write $(\alpha)_r = \alpha(\alpha + 1)\cdots (\alpha + r-1)$ for the \emph{rising factorial}.\ Let $\alpha_1,\ldots, \alpha_n$ and $\beta_1,\ldots, \beta_n$ denote complex numbers. Then the \emph{general hypergeometric differential equation}, as discussed in \cite{BeuHec}, is the equation defined by the differential operator
\[
  (\theta_K + \beta_1-1)\cdots (\theta_K + \beta_n-1) - K(\theta_K + \alpha_1)\cdots (\theta + \alpha_n).
\]
If the numbers $\beta_1,\ldots, \beta_n$ are distinct mod $\ZZ$, then $n$ independent solutions of the hypergeometric equation are given by the series
\[
  K^{1-\beta_i}\hspace{0mm} _nF_{n-1}(1 + \alpha_1-\beta_i,\ldots, 1 + \alpha_n-\beta_i; 1 + \beta_1 - \beta_i, \stackrel{\vee}{\ldots}, 1 + \beta_n-\beta_i; z),
\]
where the $\vee$ denotes omission of $1 = 1 + \beta_i-\beta_i$.

Let us apply this to our vvmf $F$ of (minimal) weight $k$.\ The eigenvalues of $\rho(T)$, which encode the local monodromy of (\ref{eq:hypgeom}) about $K = 0$, are distinct roots of unity.\ Let $r_1$ and $r_2$ denote the exponents of these eigenvalues.\ That is, $\rho(T)$ is conjugate to the matrix
\[
  \twomat{e^{2\pi i r_1}}{0}{0}{e^{2\pi i r_2}}.
\]
While these exponents are only defined mod $\ZZ$, it is natural to take them in the range $[0,1)$ for the following reason: Theorem 1.3 of \cite{MarMas} shows that since $\cH(\rho)$ is a cyclic $\CC[E_4,E_6]\langle D\rangle$-module in this case, this choice of exponents ensures that they agree with the indicial roots of the form $F$ of minimal weight. Said differently, when $\rho(T)$ is diagonal as above, the coordinates of $F$ have $q$-expansions of the form $q^{r_i}$ plus higher order terms.
 
The coordinates of $G = \eta^{-2k}F$ are now solutions of an equation (\ref{eq:hypgeom}) with indicial roots $r_1-\frac k{12}$ and $r_2-\frac k{12}$. The indicial polynomial in question is
\[
  \left(\theta_K - r_1 + \frac k{12}\right)\left(\theta_K - r_2 + \frac k{12}\right) = \theta_K^2 - \frac{1}{6}\theta_K + a.
\]
It follows that
\begin{align*}
  k &= 6(r_1 + r_2) - 1,\\
  a & = \left(r_1 - \frac k{12}\right)\left(r_2 - \frac k{12}\right).
\end{align*}
To solve equation (\ref{eq:hypgeom}) we express it in the form given by Beukers-Heckmann \cite{BeuHec}:
\[\left(\left(\theta_K +\frac{r_1-r_2}{2} +\frac{11}{12}-1\right)\left(\theta_K +\frac{r_2-r_1}{2} +\frac{11}{12}-1\right) - K\theta_K\left(\theta_K + \frac 13\right)\right)f = 0.\]
Since $\rho(T)$ has distinct eigenvalues, $r_1 - r_2$ is not an integer, and a basis of solutions to (\ref{eq:hypgeom}) near $K = 0$ is given in terms of hypergeometric series. It follows that, at the possible cost of replacing $\rho$ by an equivalent representation to account for the choice of a particular basis of solutions to equation (\ref{eq:hypgeom}), one can choose a minimal weight form $F = \hspace{0mm}^t(f_1,f_2)$ for $\rho$ so that 
\begin{align*}
f_1 &= \eta^{2k}K^{\frac{6(r_1-r_2)+1}{12}}\hspace{0mm}_2F_1\left(\frac{6(r_1-r_2)+1}{12},\frac{6(r_1-r_2)+5}{12};r_1-r_2+1;K\right),\\
f_2 &= \eta^{2k}K^{\frac{6(r_2-r_1)+1}{12}}\hspace{0mm}_2F_1\left(\frac{6(r_2-r_1)+1}{12},\frac{6(r_2-r_1)+5}{12};r_2-r_1+1;K\right).
\end{align*}
These identities were exploited in \cite{FraMas} to study arithmetic properties of Fourier coefficients of vvmfs.

\subsection{Numerical examples with $_2F_1$}
\label{ss:2f1}
Connections between hypergeometric series and modular forms have appeared in many places in the literature.\ Rather than provide an exhaustive survey, in this section we give one simple example and then explain how several known results can be viewed through the optic of vvmfs.\ For notation and results concerning Eisenstein series, consult Appendix \ref{a:eis}.

\begin{ex}
$\Gamma$ acts on Eisenstein series for $\Gamma(4)$ through the quotient $\Gamma/\pm \Gamma(4) \cong S_4$.\ This action coincides with the permutation action of $\Gamma$ on the six cusps $0,~ \frac 12,~ 1,~2,~3,~ \infty$ of $\Gamma(4)$.\ The cusp permutation representation decomposes into irreducibles of dimension $1$, $2$ and $3$.\ In Section \ref{ss:3f2} we treat the $3$-dimensional irreducible; here we discuss the $2$-dimensional irreducible.\ This $2$-dimensional $\rho$ is spanned by $b_1 = [0]+[\frac 12]-2[1]+[2]-2[3]+[\infty]$ and $b_2 = [\frac 12]-[1]-[3]+[\infty]$. In the basis $(b_1,b_2)$ one has
\[
  \rho(T) = \twomat{-2}{-1}{3}{2},
\]
and hence $r_1 = 0$, $r_2 = \frac 12$. This shows that the minimal weight is $2$.\ Thus, the series
\begin{align*}
f_1 &= \eta^{4}K^{- \frac{1}{6}}\hspace{0mm}_2F_1\left( - \frac{1}{6},\frac{1}{6};\frac 12;K\right),&
f_2 &= \eta^{4}K^{\frac{1}{3}}\hspace{0mm}_2F_1\left(\frac{1}{3},\frac{2}{3};\frac 32;K\right),\\
 &= 1 + 24q + 24 q^2 + 96q^3 + 24 q^4 + \cdots, &&=q^{1/2}(1 + 4q + 6q^{2} + 8q^{3} + 13q^{4} + \cdots),
\end{align*}
define two modular forms of weight $2$ on $\Gamma(4)$, and they make up the coordinates of a minimal weight form in $\cH(\rho')$ for a representation $\rho'$ equivalent to $\rho$. Note that $f_1$ spans the $1$-dimensional space of modular forms of weight $2$ on $\Gamma_0(2)$.

Let $g_1$ and $g_2$ denote the linear combinations of the Eisenstein series $G_P$ on $\Gamma(4)$ corresponding to the basis elements $b_1$ and $b_2$ (cf. Appendix \ref{a:eis}). One computes that
\begin{align*}
  g_1 &= \frac 12 -12q_4^2+12q_4^4-48q_4^6+12q_4^8-72q_4^{10}+48q_4^{12} -96q_4^{14}+\cdots,\\
  g_2 &= -8q_4^2-32q_4^6-48q_4^{10}-64q_4^{14}-104q_4^{18}-96q_4^{22}-112q_4^{26} + \cdots,
\end{align*}
and hence $f_1 = 2g_1-3g_2$ and $f_2 = (-1/8)g_2$.
\end{ex}

\begin{ex} 
For each integer $k$, consider the following MLDE, which is studied in the papers \cite{KanKoi1}, \cite{KanKoi2} and \cite{Tsu}:
\begin{eqnarray}\label{KZE}
\label{MLDEk}
\left(D_{k}^2 - \frac{k(k+2)E_4}{144}\right)f=0.
\end{eqnarray}
 The nature of the vvmf $F = {}^t(f_1,f_2)$ determined by (\ref{KZE}) depends very much on the congruence class of $k\pmod 6$.\
Let $\rho$ denote the representation furnished by the solutions of (\ref{KZE}).\ Over the $K$-line (upon passing to weight zero by dividing by $\eta^{2k}$, as  already discussed), it becomes the hypergeometric differential equation (\ref{eq:hypgeom}) with $a= -k(k+2)/144$ and indicial roots $(k+2)/12, -k/12$.\ The indicial roots of (\ref{KZE}) are therefore $(k+1)/6$ and $0$, and the eigenvalues of $\rho(T)$ are $e^{2\pi i(k+1)/6}$ and $1$ respectively\footnote{In this case the roots need not lie between $0$ and $1$, as we are dealing with forms that need not be of minimal weight for the representation $\rho$ furnished by the solutions of (\ref{KZE}).}.\ The indicial roots differ by an integer if and only if $k \equiv -1 \pmod{6}$. If this is 
\emph{not} satisfied then we obtain the following basis of solutions to (\ref{KZE}):
\begin{align*}
f_1 &= \eta^{2k}K^{\frac{k+2}{12}}\hspace{0mm}_2F_1\left(\frac{k+2}{12},\frac{k+6}{12};\frac{k+7}{6};K\right),\\
f_2 &= \eta^{2k}K^{\frac{-k}{12}}\hspace{0mm}_2F_1\left(\frac{-k}{12},\frac{4-k}{12};\frac{5-k}{6};K\right).
\end{align*}

\emph{Case 1}: $k+1\equiv 2, 3, 4\pmod 6$.\ In this case $\rho(T)$ has order $N = 2$ or $3$, whence $\ker \rho$ contains $\Gamma(N)$ and the solutions of (\ref{KZE}) are scalar forms of level $N$ and weight $k$.

\emph{Case 2}: $k+1\equiv \pm 1 \pmod 6$.\ Subtleties arise in this case because $\rho$ is \emph{indecomposable} and \emph{not irreducible} (\cite{Mas1},  \cite{MarMas}).\ We may assume  that 
$\rho$ is \emph{upper triangular}, so that
\begin{eqnarray}\label{indec}
\rho(\gamma) = \left(\begin{array}{cc}\alpha(\gamma) & \beta(\gamma) \\0 & \delta(\gamma)\end{array}\right)\ (\gamma \in \Gamma)
\end{eqnarray}
where $\{\alpha, \delta\}$ are the $1$-dimensional representations $\{1, \chi^{2(k+1)}\}$ ($\chi$  as in (\ref{eq:chidef})).\ The two choices
$\alpha=1$ or $\delta=1$ furnish \emph{inequivalent} representations of $\Gamma$, as do the choices of sign in the congruence for $k+1$.\
Thus $\rho$ may be any one of four representations of $\Gamma$.\
Let $\rho'$ denote one of these four representations, and let $F'$ be a nonzero holomorphic
 vvmf in $\mathcal{H}_{k_0}(\rho')$ of minimal weight $k_0$, say.\ The different choices of $\rho'$ are distinguished as follows (\cite{MarMas}, Section 4): 
 \begin{enumerate}
\item[(a)] $k_0=0, F'=\ ^t(1, 0),\ ^t(g_1, \eta^{4})$ is a free basis of $\mathcal{H}(\rho)$, $\alpha=1, k+1\equiv 1\pmod 6$;
\item[(b)] $k_0=0, F'=\ ^t(1, 0),\ ^t(g_1, \eta^{20})$ is a free basis of $\mathcal{H}(\rho)$, $\alpha=1, k+1\equiv -1\pmod 6$;
\item[(c)] $F', DF'$ is a free basis of $\mathcal{H}(\rho)$, $k_0=0$,  $\delta=1, k+1\equiv 1 \pmod 6$;
\item[(d)] $F', DF'$ is a free basis of $\mathcal{H}(\rho)$, $k_0=4$,  $\delta=1, k+1\equiv -1 \pmod 6$.
\end{enumerate}  

Using this information, it follows that $\rho$ corresponds to cases (c) or (d).\ Let $F\df {}^t(f_1, f_2)\in \cH_k(\rho)$ correspond to a fundamental system of solutions of (\ref{MLDEk}).\ Because 
$\delta=1$, the functional equation $\rho(\gamma)F = F|_{k}\gamma$ tells us that $f_2$ is a weight $k$ scalar
form on $\Gamma$.\ It is easy to see that up to normalization, $f_2(\tau)$ is the unique solution of (\ref{MLDEk}) that is a scalar form.\ This form is denoted $F_k$ by Kaneko-Zagier (\cite{KanZag}, especially  Section 8) and they show that for $p=k+1$ a prime larger than $5$,  $F_k$ is related to supersingular elliptic curves (mod $p$).\ More precisely, choose integers $m$, $\delta$ and $\veps,$ where $0 \leq \delta \leq 3$ and $0\leq \veps \leq 2$, so that $F_k\Delta^{-m}E_4^{-\delta}E_6^{-\veps}$ is of weight zero.\ This means that one can write
\[F_k\Delta^{-m}E_4^{-\delta}E_6^{-\veps} = f(j)\]
for a rational function $f(j) \in \QQ(j)$.\ Kaneko-Zagier show that if $p > 5$ is prime and $k = p-1$, then the mod $p$ reduction of $f$ is equal (up to a simple factor) to the polynomial over $\FF_p$ encoding the supersingular $j$-invariants (recall that there are finitely many such $j$-invariants and they are all defined over $\FF_{p^2}$).\ This example of Kaneko-Zagier was instrumental in the author's understanding, exploited in \cite{FraMas},  of the importance of hypergeometric series for the arithmetic of vvmfs.

\begin{rmk}
Remark 2 on page 151 of \cite{KanKoi1} observes that $f_1$ does not appear to be a modular form, based on computational evidence suggesting that it has unbounded denominators.\ That this coordinate is not modular indeed follows from the fact that $F$ is a vvmf for an \emph{indecomposable} representation of $\Gamma$ that is not irreducible: if $f_1$ was also modular, then the representation corresponding to $F$ would necessarily contain a congruence subgroup in its kernel, and it would thus necessarily be a completely reducible representation.\ Moreover, the fact that the Fourier coefficients of $f_1$ do indeed have unbounded denominators follows from \cite{FraMas}.
\end{rmk}

\emph{Case 3}: $k+1\equiv 0\pmod 6$.\ Here, $\rho(T)$ has both eigenvalues equal to $1$, so that $\rho$ is equivalent to the
standard $2$-dimensional representation of $\Gamma$ (denoted by $M_1$ in Section \ref{SSDandP}).\ Since $\rho(T)$ is not semisimple we are here confronted with a case when the associated vvmfs
will be logarithmic.\ Indeed, a nonvanishing vvmf of least weight $k_0 =-1$ is $F_0:=\ ^t(\tau, 1)$.\
The free module theorem still applies in this situation (cf. \cite{KnoMas5}), and we can take as free basis for $\mathcal{H}(\rho)$ the vvmfs 
\[F_0,\quad\quad D_{-1}F_0=\ ^t\left(\frac{1}{2\pi i}+\frac{\tau}{12}E_2,\frac{1}{12}E_2 \right).\] 
Thus $\mathcal{H}_{\ell}(\rho)=0$ for \emph{even} $\ell$, while for $m\geq 0$, $\mathcal{H}_{2m-1}(\rho)$ is spanned by vvmfs of the shape
\begin{eqnarray*}
\left(\begin{array}{c}\frac{1}{2\pi i}G_{2m-2}+\tau(F_{2m}+\frac{1}{12}G_{2m-2}E_2) \\ F_{2m}+\frac{1}{12}G_{2m-2}E_2\end{array}\right)
\end{eqnarray*}
(compare with (\ref{logvvmf})) where $F_{2m}, G_{2m-2}$ are scalar forms of weights $2m$ and $2m-2$ respectively.\ We can draw several conclusions from this.\ First,
 \emph{every} vvmf of weight $2m-1$ associated to $\rho$ has as second component a quasimodular form of weight $2m$ and depth at most $1$.\
 In particular, this applies to the solutions of  (\ref{KZE}), in which  case Kaneko-Koike \cite{KanKoi1} give an explicit formula for the quasimodular form.\ On the other hand, the construction shows that \emph{every} quasimodular form of depth at most $1$ occurs in this way, and defines
 a  \emph{bijection} between quasimodular forms of weight $2m$ and depth at most $1$ and vvmfs of weight $2m-1$ associated to $\rho$.
\end{ex}

\subsection{The monic MLDE of degree $3$}
\label{ss:solvingmlde}
Let $\rho$ be a $3$-dimensional irreducible representation of $\Gamma$ such that $\rho(T)$ has finite order.\ In this case $\rho(T)$ necessarily has distinct eigenvalues (\cite{Mar}, \cite{TubWen}).\ Let $F,G,H$ be a free $\CC[E_4,E_6]$-basis for $\cH(\rho)$ of weights $k \leq l \leq m$, respectively.\ By the discussion of Section \ref{SSMLDE}, since $\rho$ is irreducible $F$ does not satisfy an MLDE of order \emph{less} than $3$, in particular $DF\not= 0$.\ Since there are no holomorphic modular forms of weight $2$, it follows that $DF$ is a nonzero constant multiple of $G$ or $H$.\ This also shows that if $k=l$ then $F, G, DF, DG$ are linearly independent over $\CC[E_4, E_6]$, an impossibility because the rank is $3$.\ So we may take $G=DF$, in particular $l=k+2$.\ Now $D^2F$ cannot be a combination of $F$ and $DF$ with $\CC[E_4, E_6]$-coefficients (otherwise it solves an MLDE of order $2$), so we must have $D^2F=\alpha H+\beta E_4F$ for $\alpha, \beta \in \CC, \alpha\not= 0$.\ Then we can replace $H$ by $D^2F$ if necessary, and in this way we see that $F, DF, D^2F$ is a basis of $\mathcal{H}(\rho)$. As in the $2$-dimensional case, we deduce that the coordinates of $F$ satisfy an MLDE of the form
\begin{equation}
\label{eq:deg3}
(D^3_k+a E_4D_k+b E_6)f=0.
\end{equation}
When $k = 0$ this corresponds to the generalized hypergeometric equation
\[
\left(\theta_K^{3} - \left(\frac{2K+1}{2(1-K)}\right) \theta_K^{2}  + \left(\frac{18a+1-4K}{18(1-K)} \right) \theta_K + \frac{b}{1-K}\right)f = 0.
\]

Let $r_1,r_2$ and $r_3$ denote the exponents of $\rho(T)$, where as in the $2$-dimesional case we take $0 \leq r_i < 1$ for all $i$ (cf.\ Theorem 1.3 of \cite{MarMas}).\ The roots of the indicial equation near $K = 0$ of the differential equation satisfied by $G = \eta^{-2k}F$ are then $r_i - \frac{k}{12}$, so that
\[
  \left(\theta_K - r_1+\frac{k}{12}\right)\left(\theta_K - r_2+\frac{k}{12}\right)\left(\theta_K - r_3+\frac{k}{12}\right) = \theta_K^3 -\frac{1}{2}\theta_K^2 + \left(a + \frac{1}{18}\right)\theta_K + b.
\]
This shows that $k = 4(r_1+r_2+r_3)-2$ and one can similarly solve for $a$ and $b$ in terms of the exponents. Three linearly independent solutions of equation (\ref{eq:deg3}) are given in terms of generalized hypergeometric series as follows:
\begin{equation}
\label{hypgeomser}
\eta^{2k}K^{\frac{a_i+1}{6}}\hspace{0mm} _3F_{2}\left(\frac{a_i+1}{6},\frac{a_i+3}{6}, \frac{a_i+5}{6}; r_i-r_j+1, r_i-r_k+1; K\right),
\end{equation}
for $i = 1,2,3$, where for $\{i,j,k\} = \{1,2,3\}$ we write $a_i = 4r_i-2r_j-2r_k$. After possibly exchanging $\rho$ with an equivalent representation, these series form the coordinates of a nonzero vvmf of lowest weight $k = 4(r_1+r_2+r_3)-2$ for $\rho$.
\begin{rmk}
For $\rho$ of dimension $4$ and higher it need not be true that there exists a free basis for $\cH(\rho)$ of the form $F, DF,\ldots, D^{n}F$. This complicates the matter of computing vvmfs in higher dimensions.
\end{rmk}

\subsection{Numerical examples with $_3F_2$}
\label{ss:3f2}
We consider some explicit $3$-dimensional irreducible representations of $\Gamma$.\ Fix an integer $N \geq 1$ and let $\bar{\Gamma}(N):=\Gamma/\pm\Gamma(N)$.\ If $A \in \Gamma$ then we write $\bar{A}$ for its image in $\bar\Gamma(N)$.\ For $N=3, 4, 5, 7$ the quotient $\bar\Gamma(N)$ is isomorphic to $A_4, S_4, A_5, \PSL_2(7)$ respectively, the latter group being the simple group of order $168$.\ Each of these groups have faithful irreducible representations $\rho$ of dimension $3$, and we seek the least integer $k_0$ (necessarily positive and even) for which $\rho$ is realized by the action of $\bar{\Gamma}(N)$ on a space of holomorphic modular forms on $\pm \Gamma(N)$ of weight $k_0$.

The group $\Gamma(N)$ has $(N^2-1)/2$ cusps when $N$ is an odd prime, and $6$ cusps when $N=4$. In each case the quotient $\bar{\Gamma}$ permutes these cusps transitively. Thus, for $k\geq 4$ an even integer, the dimension of the space $\Eis_k(N)$ of weight $k$ Eisenstein series of level $N$ is $(N^2-1)/2$ or $6$, while for $k=2$ it is $(N^2-3)/2$ or $5$. If $k\geq 4$, there is a $1$-dimensional subspace in $\Eis_k(1)$ spanned by the Eisenstein series of level $1$.\ So for all even $k\geq 2$  there is a subspace $V_k(N)\subseteq \Eis_k(N)$ of dimension $(N^2-3)/2$ or $5$ that admits a natural action of $\bar{\Gamma}(N)$ and contains no $\bar{\Gamma}(N)$-invariants. Considered as a $\bar{\Gamma}(N)$-module, the isomorphism class of $V_k(N)$ is independent of $k$. 

When $N$ is an odd prime, the action of $\bar{\Gamma}(N)$ on the cusps is the same as its conjugation action on the Sylow $N$-subgroups of $\bar{\Gamma}(N)$ (this is because both actions are transitive, and $\bar{T}$ generates both a Sylow $N$-subgroup and the stabilizer of the infinite cusp).\ Since the normalizer of a Sylow $N$-subgroup in $\bar{\Gamma}(N)$ has order $N(N-1)/2$, this means that the trace of $\bar T$ acting on cusps is just $(N-1)/2$. Its trace on $V_k(N)$ is therefore $(N-3)/2$.\ Using this information, we can decompose $V_k(N)$ into irreducible $\bar{\Gamma}(N)$-modules, at least for the odd prime values of $N$ in front of us. When $N=4$, note that $\bar T$ has trace $1$ on $V_k(N)$.

\begin{itemize}
\item $N = 3$:\ In this case $V_k(3)$ is the unique irreducible of $A_4$ of dimension $3$.

\item $N = 4$: Here $\dim V_k(4)=5$ and $\Tr_{V_k(N)} \bar T =1$.\ The nontrivial irreducibles for $\bar{\Gamma}(4) \cong S_4$ are $1, 2, 3_1, 3_2$, where $\bar T$ has trace $-1, 0, 1, -1$ respectively.\ Thus $V_k(4)\cong 2\oplus 3_1$.\ The dimension of the space $M_4(4)$ of holomorphic modular forms of level $4$ and weight $4$ is $9$, and
\[
  M_4(4) = S_4(4) \oplus V_4(4)\oplus \Eis_4(1)
\]
with $S_4(4)$  the $3$-dimensional space of cusp-forms on $\Gamma(4)$.\ In its action on $S_4(4)$,  $\bar T$ has no eigenvalue $1$, and not all eigenvalues are $-1$  (there are no nonzero cusp-forms on $\Gamma_1(4)$ or $\Gamma(2)$ of weight $4$) so its eigenvalues must be $\pm i, -1$, with trace $-1$. Hence $S_4(4)\cong 3_2$.

\item $N = 5$: Here $\dim V_k(5)=11$ and $\Tr_{V_k(5)}\bar T=1$.\ The nontrivial irreducibles for $\bar{\Gamma}(5) \cong A_5$ are $3_1, 3_2, 4, 5$, and $\bar T$ has trace $\alpha, \bar{\alpha}, -1, 0$ respectively ($\alpha+\bar{\alpha}=1$).\ Thus $V_k(5) = 3\oplus \bar{3}\oplus 5$.

\item$N = 7$: Here $\dim V_k(7) =23$ and $\Tr_{V_k(7)}\bar T=2$.\ The nontrivial irreducibles of $\bar{\Gamma}(7) \cong \PSL_2(7)$ are $3_1, 3_2, 6, 7, 8$, and $\bar T$ has trace $\beta, \bar{\beta}, -1, 0, 1$ respectively ($\beta+\bar{\beta}=-1$). So $V_k(7) = 7\oplus 8\oplus 8$ contains no $3$-dimensional irreducibles.\  On the other hand $\Gamma(7)$ has genus $3$, so $\dim S_2(7) = 3$ and the space of cusp-forms is irreducible for $\bar{\Gamma}(7)$ (since it contains no $\Gamma$-invariants).\ This is the representation $\rho$ denoted $3_1$.\ If $\rho(\bar T)=\ddiag(e^{2\pi ir_1}, e^{2\pi ir_2}, e^{2\pi ir_3})$ then we see that $r_1+r_2+r_3\in \ZZ$, and since $\rho(\bar T)$ has order $7$ then $(r_1, r_2, r_3)=(1/7, 2/7, 4/7)$ or $(r_1,r_2,r_3) = (3/7,5/7,6/7)$ (cf. \cite{Mar} or \cite{TubWen}). The minimal weights $k_0$ are $2$ and $6$ in these cases, respectively. We will call the first case $3_1$ and the second case $3_2$.
\end{itemize}

For ease of reference we summarize the information above in a table (recall that $k_0 = 4(r_1+r_2+r_3)-2$).
\begin{center}\begin{tabular}{c|c|c|c|c|c}
$N$&$\rho$ & $k_0$ & $r_1$ & $r_2$ & $r_3$ \\
\hline
3&3&2&0&1/3&2/3\\
4&$3_1$&2&0&1/4&3/4\\
4&$3_2$&4&1/2&1/4&3/4\\
5&$3_1$&2&0&1/5&4/5\\
5&$3_2$&2&0&2/5&3/5\\
7&$3_1$&2&1/7&2/7&4/7\\
7&$3_2$&6&3/7&5/7&6/7
\end{tabular}
\end{center}
We next show how the vvmfs corresponding to these representations may be used to express certain Eisenstein series (cf. Appendix \ref{a:eis} for notation) in terms of hypergeometric functions, to compute certain linear combinations between partial zeta values, and to prove surprising polynomial identities between hypergeometric series.

\begin{ex}[The case $N = 3$, $k_0 = 2$.]
\label{ex:N3}
In this case the coordinates of a lowest weight form furnish a basis for the space of modular forms of weight $2$ for $\Gamma(3)$:
\begin{align*}
f_1 &=1728^{1/6}\eta^4K^{-1/6}\ _{3}F_2(-1/6, 1/6, 1/2; 2/3, 1/3; K)\\
&=1 + 12q + 36q^{2} + 12q^{3} + 84q^{4} + 72q^{5} + 36q^{6} +\cdots,\\
f_2 &=1728^{-1/6}\eta^4K^{1/6}\ _{3}F_2(1/6, 1/2, 5/6; 4/3, 2/3; K)\\
&=q^{\frac 13}(1 + 7q + 8q^{2} + 18q^{3} + 14q^{4} + 31q^{5} + 20q^{6} +\cdots),\\
f_3 &=1728^{-1/12}\eta^4K^{1/2}\ _{3}F_2(1/2, 5/6, 7/6; 5/3, 4/3; K)\\
&=q^{\frac 23}(1 + 2q + 5q^{2} + 4q^{3} + 8q^{4} + 6q^{5} + 14q^{8} + \cdots).
\end{align*}
Representatives for the cusps are given by $0$, $1$, $2$ and $\infty$.\ A basis for the $3$-dimensional irreducible in the corresponding permutation representation of the cusps is given by $([0] - [\infty], [1]- [\infty], [2] - [\infty])$, where the square brackets denote the $\Gamma(N)$-equivalence class of a cusp. Another basis for $M_2(\Gamma(3))$ is given by the Eisenstein series:
\begin{align*}
  g_1 &= G_{0} - G_{\infty} = \frac{1}{3} + q_N + 3q_N^2 + 4q_N^3 + 7q_N^4 + 6q_N^5+12q_N^6 + 8q_N^7+15q_N^8+\cdots,\\
  g_2 &= G_{1} - G_{\infty} = (1-\zeta_N)q_N + 3(2 + \zeta_N)q_N^2 + 7(1-\zeta_N)q_N^4+6(2+\zeta_N)q_N^5 + \cdots,\\
  g_3 &= G_{2} - G_{\infty} = (2+\zeta_N)q_N + 3(1 - \zeta_N)q_N^2 + 7(2+\zeta_N)q_N^4+6(1-\zeta_N)q_N^5 + \cdots,\\
\end{align*}
where $\zeta_N = e^{2\pi i/N}$. One deduces that:
\begin{align*}
  f_1 &= 3g_1-g_2-g_3, & f_2 &=3^{-1}((\zeta_N+1)g_2-\zeta_Ng_3), & f_3 &=3^{-2}(-\zeta_Ng_2+(\zeta_N+1)g_3).
\end{align*}
Inverting these identities allows one to express these linear combinations of Eisenstein series in terms of $\eta$, $K$ and generalized hypergeometric series.\ Note also that the constant term of $g_1$ is a partial zeta value.\ By expressing $g_1$ in terms of $f_1$, $f_2$ and $f_3$ using only the higher coefficients of these forms, one can easily solve for the value of this partial zeta function (of course there are other ways to compute such quantities).\ We give a more exciting example of this sort in Example \ref{ex:N5} below.
\end{ex}

\begin{ex}[The cases $N = 4$, $k_0 = 2,4$] 
\label{ex:N4}
The coordinates of a lowest weight form for the representation $3_1$ yields the following weight-$2$ modular forms:
\begin{align*}
f_1 &= 1728^{1/6}\eta^{4}K^{-1/6}\  _{3}F_2(-1/6, 1/6, 1/2; 3/4, 1/4; K),\\
f_2&=1728^{-1/12}\eta^{4}K^{1/12}\  _{3}F_2(1/12, 5/12, 3/4; 5/4, 1/2; K),\\
f_3&=1728^{-7/12}\eta^{4}K^{7/12}\  _{3}F_2(7/12, 11/12, 5/4; 7/4, 3/2; K).
\end{align*}
Representatives for the cusps of $\Gamma(4)$ are given by $0$,$1/2$, $1$,$2$,$3$ and $\infty$. A basis for the $3$-dimesional irreducible in the corresponding permutation representation is given by $([0]-[2],[1/2]-[\infty],[1]-[3])$.  The corresponding Eisenstein series are:
\begin{align*}
  g_1 &= G_{0} - G_{2}, & g_2 &= G_{1/2} - G_{\infty}, &  g_3 &=G_{1} - G_{3}.
\end{align*}
One computes that
\begin{align*}
  f_1 &= 2g_1,& f_2 &= 2^{-2}(g_2 - \zeta_N^{-1}g_3), &f_3 &= 2^{-3}(g_2 + \zeta_N^{-1}g_3).
\end{align*}
Note that $g_1$ is the Eisenstein series of weight $2$ on $\Gamma_0(4)$.

The coordinates of a lowest weight form for the representation $3_2$ yields the following weight-$4$ modular forms:
\begin{align*}
f_4&=1728^{-1/6}\eta^{8}K^{1/6}\ \  _{3}F_2(1/6, 1/2, 5/6; 5/4, 3/4; K) = f_2^2 - 16f_3^2,\\
f_5&=1728^{1/12}\eta^{8}K^{-1/12}\  _{3}F_2(-1/12, 1/4, 7/12; 3/4, 1/2; K)= f_1f_2,\\
f_6&=1728^{-5/12}\eta^{8}K^{5/12}\ \  _{3}F_2(5/12, 3/4, 13/12; 5/4, 3/2; K) = f_1f_3.
\end{align*}
These polynomial identities between modular forms yield the following nonobvious relations between generalized hypergeometric series.
\begin{prop}
The following identities hold:
\begin{align*}
 \ \  _{3}F_2\left(\frac 16, \frac 12, \frac 56; \frac 54, \frac 34; K\right) &= \  _{3}F_2\left(\frac{1}{12}, \frac{5}{12}, \frac{3}{4}; \frac{5}{4}, \frac{1}{2}; K\right)^2 - \frac{K}{108}\  _{3}F_2\left(\frac{7}{12}, \frac{11}{12}, \frac{5}{4}; \frac{7}{4}, \frac{3}{2}; K\right)^2,\\
\  _{3}F_2\left(-\frac{1}{12}, \frac{1}{4}, \frac{7}{12}; \frac{3}{4}, \frac{1}{2}; K\right) &= \  _{3}F_2\left(-\frac{1}{6}, \frac{1}{6}, \frac{1}{2}; \frac{3}{4}, \frac{1}{4}; K\right)\  _{3}F_2\left(\frac{1}{12}, \frac{5}{12}, \frac{3}{4}; \frac{5}{4}, \frac{1}{2}; K\right), \\
\ \  _{3}F_2\left(\frac{5}{12}, \frac 34,\frac{13}{12}; \frac 54, \frac 32; K\right)&= \  _{3}F_2\left(-\frac 16, \frac 16, \frac 12; \frac 34, \frac 14; K\right) \  _{3}F_2\left(\frac{7}{12}, \frac{11}{12}, \frac{5}{4}; \frac{7}{4}, \frac{3}{2}; K\right).
\end{align*}
\end{prop}
\end{ex}

\begin{ex}[The case $N = 5$, $k_0 = 2$]
\label{ex:N5}
In this case $M_2(5)$ contains two irreducible $3$-dimensional representations of $\Gamma$.\ The coordinates of a lowest weight form with respect to one of the representations are:
\begin{align*}
f_1&=1728^{1/6}\eta^{4}K^{-1/6}\  _{3}F_2(-1/6, 1/6, 1/2; 4/5, 1/5; K), \\
f_2&= 1728^{-1/30}\eta^{4}K^{1/30}\ \  _{3}F_2(1/30, 11/30, 7/10; 6/5, 2/5; K),\\
f_3 &=1728^{-19/30}\eta^{4}K^{19/30}\  _{3}F_2(19/30, 29/30, 13/10; 9/5, 8/5; K).
\end{align*}
The second representation yields the forms:
\begin{align*}
f_4&=1728^{1/6}\eta^{4}K^{-1/6}\  _{3}F_2(-1/6, 1/6, 1/2; 3/5, 2/5; K), \\
f_5 &=1728^{-7/30}\eta^{4}K^{7/30}\  _{3}F_2(7/30, 17/30, 9/10; 7/5, 4/5; K),\\
f_6&=1728^{-13/30}\eta^{4}K^{13/30}\  _{3}F_2(13/30, 23/30, 11/10; 8/5, 6/5; K).
\end{align*}
All of the forms above occur in the Eisenstein space of level $5$ and weight $2$.\ Representatives for the cusps are given by $0$, $2/5$, $1/2$, $1$, $3/2$, $2$, $5/2$, $3$, $7/2$, $4$, $9/2$, $\infty$.\ Let $\zeta = e^{2\pi i/5}$ denote a primitive $5$th root of unity and let $\alpha = \zeta^3+\zeta^2$.\ Then bases for the two $3$-dimensional irreducibles are given by 
\begin{align*}
(&[0]-[1]+\alpha[3/2]+\alpha[2] - [5/2]+[7/2]-\alpha [4] - \alpha [9/2],\\
&[2/5]+ \alpha [1] - [3/2] +\alpha [2] -\alpha [7/2] + [4] -\alpha [9/2] -[\infty],\\
& [1/2] - \alpha [1] -\alpha [3/2] + [2] - [3] + \alpha [7/2] + \alpha [4] - [9/2])
\end{align*}
and
\begin{align*}
(&[0]-[1]-(\alpha+1)[3/2]-(\alpha+1) [2] - [5/2]+[7/2]+(\alpha+1)[4] + (\alpha+1)[9/2],\\
&[2/5]- (\alpha+1) [1] - [3/2] -(\alpha+1)[2] +(\alpha+1) [7/2] + [4] +(\alpha+1)[9/2] -[\infty],\\
& [1/2] + (\alpha+1)[1] +(\alpha+1)[3/2] + [2] - [3] - (\alpha+1) [7/2] - (\alpha+1)[4] - [9/2]).
\end{align*}
Let $g_1$, $g_2$ and $g_3$ denote the linear combinations of Eisenstein series corresponding to the elements of the first basis, and similarly let $g_4$, $g_5$ and $g_6$ denote the linear combinations of Eisenstein series corresponding to the elements of the second basis.\ Then one has
\begin{align*}
f_1 &= -(2\zeta^3+2\zeta^2+1)g_1 - g_2 + g_3, & f_4 &= -(2\zeta^3+2\zeta^2+1)g_4+g_5 - g_6,\\
f_2 &= (1/5)(g_2-\zeta^3g_3),& f_5 &= (1/5)(g_5-\zeta g_6),\\
f_3 &=(1/15)(g_2-\zeta^2g_3),& f_6 &= (1/10)(g_5-\zeta^4g_6).
\end{align*}
The equality of constant terms in the expressions for $f_1$ and $f_4$ yields
\[
  \sum_{m \equiv 1 \pmod{5}} \frac{1}{m^2} - \sum_{m \equiv 2 \pmod{5}} \frac{1}{m^2} = \frac{(2\pi)^2}{25\sqrt{5}}.
\]
Again, there are easier ways to deduce such formulae.
\end{ex}

\begin{ex}[The cases $N = 7$, $k_0 = 2,6$]
\label{ex:N7}
The coordinates of a vvmf of lowest weight for $3_1$ furnish a basis for the space of cusp forms of weight $2$ on $\Gamma(7)$:
\begin{align*}
f_1&=1728^{1/42}\eta^{4}K^{-1/42}\  _{3}F_2(-1/42, 13/42, 9/14; 6/7, 4/7; K) \\
&=q^{\frac 17}(1 - 3q + 4q^{3} + 2q^{4} + 3q^{5} - 12q^{6} - 5q^{7} + \cdots), \\
f_2&= 1728^{-5/42}\eta^{4}K^{5/42}\  _{3}F_2(5/42, 19/42, 11/14; 8/7, 5/7; K)\\
&= q^{\frac 27}(1 - 3q - q^{2} + 8q^{3} - 6q^{5} - 4q^{6} + \cdots),\\
f_3&=1728^{-17/42}\eta^{4}K^{17/42}\  _{3}F_2(17/42, 31/42, 15/14; 10/7, 9/7; K)\\
&=q^{\frac 47}(1 - 4q + 3q^{2} + 5q^{3} - 5q^{4} - 8q^{6} + 10q^{7} +\cdots). 
\end{align*}
The modular curve $X(7)$ is the Klein quartic, and the forms above agree with the forms listed in (4.4) of \cite{Elk} up to sign. In that paper one can also find expressions for $f_1$, $f_2$ and $f_3$ in terms of theta series, and as infinite products. Further, the Klein quartic relation obtained from the canonical embedding of $X(7)$ yields
\[
  f_2^3f_1 = f_1^3f_3 + f_3^3f_2.
\]
This gives the following identity for hypergeometric series.
\begin{prop}[Klein quartic identity]
One has
\begin{align*}
& \  _{3}F_2\left(\frac{5}{42}, \frac{19}{42}, \frac{11}{14}; \frac{8}{7}, \frac{5}{7}; K\right)^3\  _{3}F_2\left(-\frac{1}{42}, \frac{13}{42}, \frac{9}{14}; \frac{6}{7}, \frac{4}{7}; K\right) = \\
&\quad \  _{3}F_2\left(-\frac{1}{42}, \frac{13}{42}, \frac{9}{14}; \frac{6}{7}, \frac{4}{7}; K\right)^3\  _{3}F_2\left(\frac{17}{42}, \frac{31}{42}, \frac{15}{14}; \frac{10}{7}, \frac{9}{7}; K\right) +\\
&\quad  \frac{K}{1728}\  _{3}F_2\left(\frac{17}{42}, \frac{31}{42}, \frac{15}{14}; \frac{10}{7}, \frac{9}{7}; K\right)^3\  _{3}F_2\left(\frac{5}{42}, \frac{19}{42}, \frac{11}{14}; \frac{8}{7}, \frac{5}{7}; K\right).
\end{align*}
\end{prop}

The coordinates of a minimal weight form for the representation $3_2$ yields the following forms of weight $6$:
\begin{align*}
f_4&=1728^{1/14}\eta^{12}K^{-1/14}\  _{3}F_2(-1/14, 11/42, 25/42; 5/7, 4/7; K) = f_1^3 + 3f_2f_3^2,\\
f_5&= 1728^{-3/14}\eta^{12}K^{3/14}\  \ _{3}F_2(3/14, 23/42, 37/42; 9/7, 6/7; K) = f_2^2f_1 - (1/3)f_3^3,\\
f_6&=1728^{-5/14}\eta^{12}K^{5/14}\ \  _{3}F_2(5/14, 29/42, 43/42; 10/7, 8/7; K)= (3/2)f_1^2f_3 - (1/2)f_2^3. 
\end{align*}
As above, these polynomial identities are equivalent to certain identities between generalized hypergeometric series.
\end{ex}

\appendix
\section{Holomorphic Eisenstein series on $\Gamma(N)$}
\label{a:eis}
This appendix recalls facts and establishes notation concerning Eisenstein series for principal congruence subgroups of $\Gamma$.\ A good reference  is \cite{Sch}.\ Let $N \geq 1$ be an integer, let $a_1, a_2$ be arbitrary integers, and write $^ta = (a_1,a_2)$.\ For integers $k \geq 3$ define
\[
  G_{N,k,a}(\tau) = \left(\frac{N}{2\pi}\right)^2\sum'_{\substack{m_1 \equiv a_1 \pmod{N}\\ m_2 \equiv a_2 \pmod{N}}} (m_1\tau + m_2)^{-k},
\]
where the prime indicates that the term $(m_1,m_2) = (0,0)$ is to be omitted.\ These series converge uniformly on compacta to define holomorphic functions on $\uhp$.\ One has the following relations:
\begin{align*}
  G_{N,k,a} &= (-1)^k G_{N,k,-a},\\
  G_{N,k,a} &= G_{N,k,b} &&\text{if } (a_1,a_2) \equiv (b_1,b_2) \pmod{N},\\
  G_{dN,k,da} &= d^{-k}G_{N,k, a} &&\text{if } d \in \ZZ_{\neq 0}.
\end{align*}
Define $G_{N,k,a}$ to be \emph{primitive} if $\gcd(a_1,a_2,N) = 1$.\ The imprimitive series can be expressed in terms of primitive series of the same level.\ For $k \geq 3$ the span of these primitive series for varying $a_1$ and $a_2$ give a basis for the orthogonal complement, under the Petersson inner-product, of the cusp forms of weight $k$ and level $N$.\ The dimension of their span is equal to the number of cusps for $\Gamma(N)$.

At weight $2$ the story is a little different, as the series above are only conditionally convergent.\ By introducing an auxilliary complex variable and analytically continuing, Hecke was able to find the following Eisenstein series in weight $2$: let
\[
  G_{N,2,a}(\tau) = (2\pi i)^{-1} (\tau - \bar \tau)^{-1} + \sum_{n \geq 0} \alpha_n(N,a) q_N^{n},
\]
where $q_N = e^{2\pi i \tau/N}$ and
\begin{align*}
  \alpha_0(N,a) &= \left(\frac{N}{2\pi}\right)^2\delta\left(\frac{a_1}{N}\right)\sum_{m_2 \equiv a_2 \pmod{N}}' m_2^{-2},\\
  \alpha_n(N,a) &= -\sum_{\substack{m \mid n\\ \frac{n}{m} \equiv a_1 \pmod{N}}}\abs{m}\zeta_N^{a_2m},\quad\quad n \geq 1,
\end{align*}
where $\zeta_N = e^{2\pi i/N}$, and $\delta(a_1/N) = 1$ if $N \mid a_1$ and it is $0$ otherwise.\ These are not holomorphic modular forms, but their differences for varying $a$ are.\ The above formula shows that they are holomorphic in $\uhp$ and at the cusp $i\infty$, and the transformation law
\[
  G_{N,k,a}|_k M = G_{N,k,^tMa}
\]
for all $M \in \Gamma(1)$, yields holomorphy of differences at the remaining cusps.\ As in the case of weight $k \geq 3$, the differences of these series in weight $2$ span the orthocomplement to the cusp forms for $\Gamma(N)$.\ This transformation law also allows one to identify the representation of $\Gamma$ spanned by these Eisenstein series with the permutation representation of $\Gamma$ acting on the cusps of $\Gamma(N)$.\ If $P = a/b$ is a cusp for 
$\Gamma(N)$ with $\gcd(a,b) = 1$, let us write $G_{N,k,P} = G_{N,k,(a,b)^t}$.\ When $N$ and $k$ are fixed we suppress them from the notation.


\end{document}